\documentclass[a4paper,12pt]{amsart}

\usepackage{amsmath,amsthm,amsfonts,amssymb,stmaryrd,mathrsfs,enumitem}
\usepackage{latexsym}
\usepackage{fullpage}
\usepackage{a4,color,palatino,fancyhdr}
\usepackage{setspace}
\usepackage{graphicx}
\usepackage{float}  
\usepackage[small, bf, margin=90pt, tableposition=bottom]{caption}
\RequirePackage{amssymb}
\RequirePackage[T1]{fontenc}
\usepackage{amscd}
\usepackage{xypic}
\input{xy}
\xyoption{all}
\xyoption{poly}
\usepackage[all]{xy}
\usepackage{tikz}
\newtheorem{theorem}{Theorem}[section]
\newtheorem{proposition}{Proposition}[section]
\newtheorem{definition}{Definition}
\newtheorem{lemma}{Lemma}[section]
\newtheorem{corollary}{Corollary}

\newtheorem{remark}{Remark}
\newtheorem{example}{Example}

\numberwithin{equation}{section}

\def\Ker{\mathop{\rm Ker}\nolimits}

\def\Aut{{\mathop{\rm Aut}}}

\begin{document}

\title{A theory of orbit braids}
\author {Hao Li,  Zhi L\"u$^*$ and Fengling Li$^\dag$}

 \subjclass[2010]{}
\thanks{$^*$Partially supported by the  NSFC  grants (No. 11661131004 and 11431009).}
\thanks{$^\dag$Partially supported by the  NSFC  grant (No. 11671064).}
\keywords{Orbit braid, orbit configuration space, extended fundamental group.}
\address{School of Mathematical Sciences, Fudan University, Shanghai, 200433, P. R. China. }
\email{14110840001@fudan.edu.cn}

\address{School of Mathematical Sciences, Fudan University, Shanghai, 200433, P. R. China. }
 \email{zlu@fudan.edu.cn}

\address{School of Mathematical Sciences, Dalian University of Technology, Dalian, 116024, P. R. China. }
 \email{fenglingli@dlut.edu.cn}

\begin{abstract}

This paper upbuilds the theoretical framework of orbit braids in $M\times I$ by making use of the orbit configuration space $F_G(M,n)$, which enriches the theory of ordinary braids, where  $M$ is a connected topological manifold of dimension at least 2 with an effective action of a finite group $G$ and the action of $G$ on $I$ is trivial. Main points of our work include:
\begin{itemize}
 \item[$\bullet$]  The notion of orbit braids is first given in the sense of Artin. However, the equivariant isotopy cannot be used as the equivalence relation among orbit braids in general. Based upon the nature of orbit braids, we define an equivalence relation among all orbit braids at an orbit base point, so that all equivalence classes can form a group denoted by $\mathcal{B}_n^{orb}(M,G)$, called the {\em orbit braid group}.
        We show that $\mathcal{B}_n^{orb}(M,G)$ is isomorphic to  a group  with an additional endowed operation (called the {\em extended fundamental group}),
        formed by   the homotopy classes of some paths (not necessarily closed paths)  in $F_G(M,n)$, which is an essential extension for fundamental groups. 
 \item[$\bullet$] The orbit braid group $\mathcal{B}_n^{orb}(M,G)$  is  large enough to contain the fundamental group of $F_G(M,n)$ and other various braid groups as its subgroups. Around the central position of $\mathcal{B}_n^{orb}(M,G)$, we obtain five short exact sequences weaved in a commutative diagram, one of which is
     $$1\longrightarrow \pi_1(F_G(M,n))\longrightarrow\mathcal{B}_n^{orb}(M, G)\longrightarrow G^{\times n} \rtimes \Sigma_n\longrightarrow 1$$
     but it is not induced by some known fibration except that the $G$-action is free. We also analyze the essential relations among various braid groups associated to those configuration spaces $F_G(M,n), F(M/G,n)$, and $F(M,n)$.
 \item[$\bullet$]  We finally consider how to give the presentations of orbit braid groups in terms of orbit braids as generators. We carry out our work by choosing $M=\mathbb{C}$ with  typical actions of  $\mathbb{Z}_p$ and $(\mathbb{Z}_2)^2$. We obtain the presentations of the corresponding orbit braid groups, from which we will see that the generalized braid groups  (introduced by Brieskorn)  corresponding to the Coxeter groups $B_n$ and $D_n$ can be
     the orbit braid group $\mathcal{B}_n^{orb}(\mathbb{C}\setminus\{0\}, \mathbb{Z}_2)$ and a subgroup of $\mathcal{B}_n^{orb}(\mathbb{C}, \mathbb{Z}_2)$, respectively.
\end{itemize}
In addition, the notion of  extended fundamental groups is also defined in a general way in the category of topology and some  characteristics  extracted from the discussions of orbit braids are given.
\end{abstract}
\maketitle


\section{Introduction}
Braid groups are fundamental objects in mathematics, which were first defined rigorously and studied by Artin in 1925 (\cite{A1}, also see \cite{A2}), although they already implicitly appeared in  the works of Hurwitz~\cite{H} in 1891 and Fricke--Klein~\cite{FK} in 1897, as Magnus~\cite{M} pointed out in 1974.  The subject has  continued to further develop and flourish by extending  ideas of braid groups or combining with various
ideas and theories from other research areas since then. For example, Fox and Neuwrith~\cite{FN} gave an alternative description of the classical braid groups by using the fundamental group  of (unordered) configuration spaces. Brieskorn~\cite{B1} extended the notion to Artin groups or the generalized braid groups by associating to all finite Coxeter groups.

\vskip .2cm
Compatible with various points of view, the notion of braid groups was uniformly  defined by Vershinin~\cite{VV} in a general way as follows: Choose a connected topological manifold $\mathbb{M}$ admitting an action of a finite group $\mathbb{G}$.
Let $Y_{\mathbb{G}}$ be the subspace of $\mathbb{M}$, formed by all points  of free orbit type in $\mathbb{M}$. So the action of $\mathbb{G}$ restricted to $Y_{\mathbb{G}}$ is free.   Assume that $Y_{\mathbb{G}}$ is connected. Then there is a fibration $Y_{\mathbb{G}}\longrightarrow X_{\mathbb{G}}$ with fiber $\mathbb{G}$, which gives  a short exact sequence:
$$1\longrightarrow \pi_1(Y_{\mathbb{G}})\longrightarrow \pi_1(X_{\mathbb{G}})\longrightarrow \mathbb{G}\longrightarrow 1$$
 The fundamental group $\pi_1(X_{\mathbb{G}})$ is called the {\em braid group of the action of $\mathbb{G}$ on $\mathbb{M}$}, denoted by $Br(\mathbb{M}, \mathbb{G})$, and
 the fundamental group $\pi_1(Y_{\mathbb{G}})$ is called the {\em pure braid group of the action of $\mathbb{G}$ on $\mathbb{M}$}, denoted by $P(\mathbb{M}, \mathbb{G})$.

 \vskip .2cm
 As an example of the notion above, for a connected topological manifold $M$ of dimension greater than one, take $\mathbb{M}=M^{\times n}$ (Cartesian product of $n$ copies of $M$). Then there is a natural action of the symmetric group
 $\mathbb{G}=\Sigma_n$ on $\mathbb{M}$, defined by
 $\sigma(x_1, ..., x_n)=(x_{\sigma^{-1}(1)}, ..., x_{\sigma^{-1}(n)}), \ \ \sigma\in \Sigma_n.$
 So $Y_{\Sigma_n}$ will be the ordered configuration space
$$F(M, n)=\{(x_1, ..., x_n)\in M^{\times n}| x_i\not= x_j \text{ for } i\not=j\}$$
(introduced by Fadell and Neuwrith~\cite{FN1}) and $X_{\Sigma_n}$ will be the unordered configuration space $F(M, n)/\Sigma_n$.   Thus,
the braid group $Br(\mathbb{M}, \Sigma_n)$ is the fundamental group $\pi_1(F(M,n)/\Sigma_n)$, also simply denoted by $B_n(M)$, and the pure braid group
$P(\mathbb{M}, \Sigma_n)$ is the fundamental group $\pi_1(F(M,n))$, also simply denoted by $P_n(M)$.

\vskip .2cm
  Theory of braids obviously possesses the following basic {\em theoretical features}:
   \begin{itemize}
 \item[$({\bf F_1})$] (Topological feature) Each braid group is realized as the fundamental group of the orbit space of  a geometric object with free action of a group.
 \item[$({\bf F_2})$]  (Algebraic feature)  Each braid group uniquely corresponds to a short exact sequence induced by the geometric object with free action.
   \end{itemize}

 \vskip .2cm

In this paper we shall begin with the study of {\em orbit braids} in $M\times I$ where $M$ is a connected topological manifold of dimension at least two with an effective action of a finite group $G$ but the action of $G$ on $M$ is not necessarily assumed to be free, and the action of $G$ on $I$ is  trivial. As far as authors know, the theory with respect to orbit braids has not been founded, especially in the case of non-free $G$-actions. In our case, since the restriction of free action is broken, it should not be surprising that {\em orbit braid group} defined here in general would not possess the topological feature (${\bf F_1}$), but as we shall see  it can still be identified with a group formed by  homotopy classes of some paths (not necessarily closed paths), called the {\em extended fundamental group}. Such a group is  large enough to contain the fundamental group of $F_G(M,n)$ and other various braid groups as its subgroups, implying that each orbit braid group can  correspond to many different short exact sequences rather than uniquely one short exact sequence, so it does not possess the algebraic feature (${\bf F_2}$) yet.
Therefore, these different points extend and enrich the theory of braids and the theory of fundamental groups.
\vskip .2cm
   The objective of this paper is  to establish the theoretical framework of orbit braids.
   Our strategy to do this is to mix the original idea of Artin and the theory of transformation groups together by making use of the construction of orbit configuration spaces. Specifically, we shall perform our work on the orbit configuration space $F_G(M,n)$ around the following  questions:
 \begin{itemize}
 \item[$({\bf Q_1})$] {\em How to define the orbit braid group  formed by all orbit braids?}
 \item[$({\bf Q_2})$] {\em What is the essential connection between the orbit braid group with its subgroups and the fundamental group of $F_G(M,n)$?}
   \item[$({\bf Q_3})$]  {\em How to present the orbit braid group in terms of orbit braids as generators?}
 \end{itemize}

\vskip .2cm
We first use the paths in $F_G(M,n)$ to describe the  braids in $M\times I$. This is the starting point of our work. Then, in the sense of Artin, an orbit braid will be defined as the orbit of a braid in $M\times I$ under the action of $G$  (see Definition~\ref{orbit braid}), where the orbit of a string of the  braid gives an orbit string of the  orbit braid,  but generally each orbit string will not be the disjoint union of some ordinary strings since the action of $G$ on $M$ is not assumed to be free. Like ordinary braids, there is still a natural operation on orbit braids by gluing endpoints of different orbit braids, but the operation is not associative.

\vskip .2cm
In the theory of ordinary braids, we  see that the natural operation acting on the isotopy classes of ordinary braids is associative,  so that the isotopy classes of ordinary braids can form the required braid group.
However, in the case of orbit braids, equivariant isotopy classes of orbit braids will not work very well except  that the action of $G$ on $M$ is free.

\vskip .2cm

 Based upon the nature of orbit braids, our approach  to determine the equivalence of two orbit braids is to detect  whether there exist two isotopic ordinary braids compatible with the $G$-action in two orbit braids (see Definition~\ref{D1}). This is also equivalent to saying whether there are two homotopic paths in $F_G(M,n)$  which just define those two orbit braids  (see Proposition~\ref{sn-condition}). Thus,  a key difficult point for the equivalence relation among orbit braids is overcome.  Moreover, we show that the set of the equivalence classes of all orbit braids at a fixed orbit base point $\widetilde{c({\bf x})}$, denoted by $\mathcal{B}_n^{orb}(M, G)$, bijectively corresponds to the set of the homotopy classes of those paths with fixed starting point ${\bf x}$ and ending points laid in ${\bf x}^{orb}$  in $F_G(M,n)$, denoted by $\pi_1^E(F_G(M,n), {\bf x}, {\bf x}^{orb})$,
where ${\bf x}$ is a fixed point of free orbit type in $F_G(M,n)$ (see Corollary~\ref{one-one}).
Furthermore, we can conclude that $\mathcal{B}_n^{orb}(M, G)$  forms a group, called the {\em orbit braid group}. This provides us an insight that, by endowing an additional operation, the set $\mathcal{\pi}_1^E(F_G(M,n), {\bf x}, {\bf x}^{orb})$ corresponding to $\mathcal{B}_n^{orb}(M, G)$ also forms a group, which is called the {\em extended fundamental group} of $F_G(M,n)$ at ${\bf x}^{orb}$.

\begin{theorem}[Theorem~\ref{group'}]\label{g-f-g}
The orbit braid group $\mathcal{B}_n^{orb}(M, G)$ is isomorphic to the extended fundamental group $\pi_1^E(F_G(M,n), {\bf x}, {\bf x}^{orb})$.
\end{theorem}

Of course, this point of view can also be used in the theory of ordinary braids. Actually, when $G=\{e\}$ is trivial, $\mathcal{B}_n^{orb}(M, G)$ degenerates into the ordinary braid group $B_n(M)$, meanwhile $\pi_1^E(F_G(M,n), {\bf x}, {\bf x}^{orb})$ is exactly isomorphic to the fundamental group of $F(M,n)/\Sigma_n$.
\vskip .2cm

The orbit braid group $\mathcal{B}_n^{orb}(M, G)$ contains some interesting subgroups $\mathcal{P}_n^{orb}(M,G)$, $\mathcal{B}_n(M,G)$ and $\mathcal{P}_n(M,G)$ (see Defintion~\ref{subgroup-def}), where  $\mathcal{P}_n(M,G)$ is exactly isomorphic to the fundamental group of $F_G(M,n)$. So the fundamental group of $F_G(M,n)$ can be regarded as a subgroup of $\mathcal{B}_n^{orb}(M, G)$. Of course, it also is a subgroup of the extended fundamental group $\pi_1^E(F_G(M,n), {\bf x}, {\bf x}^{orb})$.  On the other hand, each class of $\mathcal{B}_n^{orb}(M, G)$ determines a unique pair $(g, \sigma)\in G^{\times n}\times \Sigma_n$. This leads us to obtain  an epimorphism
 $$\Phi: \mathcal{B}_n^{orb}(M, G)\longrightarrow G^{\times n} \rtimes \Sigma_n,$$
so that we can further analyze the relations among  those subgroups of $\mathcal{B}_n^{orb}(M, G)$. Our  result is then stated as follows.
 \begin{theorem}[Theorem~\ref{main result}] \label{main}
 There are five short exact sequences
around $\mathcal{B}_n^{orb}(M, G)$, which form the following commutative diagram:
$$
\xymatrix{
1 \ar[dr] & & & & 1\\
 & \mathcal{P}_n^{orb}(M,G)    \ar[rr]  \ar[dr] &   & G^{\times n} \ar[ur]  &    \\
1 \ar[r] & \mathcal{P}_n(M,G) \ar[u] \ar[d] \ar[r]  & \mathcal{B}_n^{orb}(M,G) \ar[ur] \ar[r]^{\Phi} \ar[dr] & G^{\times n}\rtimes\Sigma_n
\ar[u] \ar[r] \ar[d] & 1 \\
& \mathcal{B}_n(M,G) \ar[rr]  \ar[ur] & & \Sigma_n \ar[dr] &   \\
1 \ar[ur] &&&& 1
}
$$
 \end{theorem}
However, the natural action of $ G^{\times n} \rtimes \Sigma_n$ on $F_G(M,n)$ is non-free except that the action of $G$ on $M$ is free, so generally $\mathcal{B}_n^{orb}(M, G)$ is not realizable as the fundamental group of the orbit space $F_G(M,n)/G^{\times n} \rtimes \Sigma_n$. This also means that generally the short exact sequence
\begin{equation}\label{seq}
1 \longrightarrow \mathcal{P}_n(M,G) \longrightarrow \mathcal{B}_n^{orb}(M,G) \longrightarrow G^{\times n}\rtimes\Sigma_n \longrightarrow 1
\end{equation}
is not induced from some known fibration. Therefore, this is completely different from the theory of ordinary braids.
Only when the action of $G$ on $M$ is free,
we shall see from Corollary~\ref{free-1} that $\mathcal{B}_n^{orb}(M, G)\cong Br(M^{\times n}, G^{\times n} \rtimes \Sigma_n)$
and $\mathcal{P}_n^{orb}(M, G)\cong Br(M^{\times n}, G^{\times n})$ in the sense of Vershinin.
In addition, Corollary~\ref{free-2} will also tell us that if the action of $G$ on $M$ is free, then $\mathcal{B}_n^{orb}(M,G)\cong B_n(M/G)$ so $\mathcal{B}_n^{orb}(M,G)$ is realizable as the fundamental group of $F(M/G,n)/\Sigma_n$, and $\mathcal{P}_n^{orb}(M,G)\cong P_n(M/G)$ so $\mathcal{P}_n^{orb}(M,G)$ is realizable as the fundamental group of $F(M/G,n)$.

\vskip .2cm
   When the action of $G$ on $M$ is non-free, unfortunately we have not found a geometric model whose fundamental group can be realized as $\mathcal{B}_n^{orb}(M, G)$ or $\mathcal{P}_n^{orb}(M, G)$. Although so, we would like to propose the following problem.
    \vskip .2cm  \noindent $({\bf P})$ (Realization problem) {\em Is there a fibration that induces the short exact sequence in (\ref{seq})? }
  \vskip .2cm
We also prove that the embedding $F_G(M,n)\hookrightarrow F(M,n)$ can induce an epimorphism between two short exact sequences produced by two fibrations
$F_G(M,n)\longrightarrow F_G(M,n)/\Sigma_n$ and $F(M,n)\longrightarrow F(M,n)/\Sigma_n$ (Proposition~\ref{orbit-ordinary}).

\begin{remark}
The geometric intuition of orbit braids plays an important role in the study of problem; meanwhile  this also lets us deepen the understanding on the extended fundamental group via the study of orbit braid groups. So, as an extension of fundamental groups, the notion of  extended fundamental groups should be defined in a general way in the category of topology. We shall do this in Section~\ref{extended fg} and  some  characteristics  extracted from the discussions of orbit braids will be given therein.
  It seems  be possible that the study of  extended fundamental groups should be able to  go on  in its own way from the viewpoints of topology and algebra, but this will be a bit far from the topic of this paper. Thus  we remain this in our subsequent
work.
\end{remark}

\begin{remark}
It should be pointed out that the orbit braid groups or the extended fundamental groups do depend upon the choices of (orbit) base points, as we can see this from Remark~\ref{base point} and Theorems~\ref{efg-main}--\ref{efg-main2}  in Section~\ref{extended fg}. The reason why we  choose the point of free orbit type as the base point  in this paper is because this does not influence on revealing our idea and opinion except for more clear expression.
\end{remark}
With respect to (${\bf Q_3})$, we consider the calculation of orbit braid groups, which is also the calculation of the extended fundamental groups.
As it was well-known, Artin began with the calculation of braid groups by considering braids on $\mathbb{R}^2\times I$.
Thus we will start our work from the cases of $\mathbb{C} \approx \mathbb{R}^2$ with the following two typical actions.

\vskip .2cm
The first one is
$\mathbb{Z}_p\curvearrowright^{\phi_1}\mathbb{C}$
defined by $(e^{{2k\pi{\bf i}}\over p}, z)\longmapsto e^{{2k\pi{\bf i}}\over p}z$, which is non-free and fixes only the origin of $\mathbb{C}$, where $p$ is a prime, and
$\mathbb{Z}_p$ is regarded as the subgroup $\{e^{{2k\pi{\bf i}}\over p}|0\leq k<p\}$. If the action $\phi_1$ is restricted to $\mathbb{C}^{\times}=\mathbb{C}\setminus\{0\}$, then the action $\mathbb{Z}_p\curvearrowright^{\phi_1}\mathbb{C}^{\times}$ is free.
The other one is $(\mathbb{Z}_2)^2\curvearrowright^{\phi_2}\mathbb{C}$ defined by
$$\begin{cases}
z\longmapsto \overline{z}\\
z\longmapsto -\overline{z}
\end{cases}
$$
which is the standard representation of $(\mathbb{Z}_2)^2$ on $\mathbb{C}\approx \mathbb{R}^2$, and this action is non-free.

\vskip .2cm

 We obtain the presentations of three orbit braid groups $\mathcal{B}_n^{orb}(\mathbb{C}, \mathbb{Z}_p)$, $\mathcal{B}_n^{orb}(\mathbb{C}^{\times}, \mathbb{Z}_p)$ and $\mathcal{B}_n^{orb}(\mathbb{C}, (\mathbb{Z}_2)^2)$.

 \vskip .2cm
 For $p=2$, since $(\mathbb{Z}_2)^n\rtimes \Sigma_n$ is isomorphic to the finite Coxeter group $B_n$ and the action $\mathbb{Z}_2\curvearrowright^{\phi_1}\mathbb{C}^{\times}$ is free, $\mathcal{B}_n^{orb}(\mathbb{C}^{\times}, \mathbb{Z}_2)$
 is exactly isomorphic to the generalized braid group $Br(B_n)$ defined by Brieskorn.
This means that in this case, the orbit braid group agrees with the generalized braid group. In addition, we will see that
the generalized braid group $Br(D_n)$ is isomorphic to a subgroup of $\mathcal{B}_n^{orb}(\mathbb{C}, \mathbb{Z}_2)$.

 \vskip .2cm
  It should be pointed out that although the group $G$ is assumed to be finite,  many aspects of our work do not need this restriction. This can be seen in Section~\ref{extended fg}.

  \vskip .2cm

  The paper is organized as follows. Section~\ref{theory} is the main part of this paper, where we will discuss how to upbuild the theoretical framework of orbit braids.
  We will begin with the basic notion of orbit braids and the definition of the equivalence relation among orbit braids. Then we give the definitions of the orbit braid group and the extended fundamental group, and show the equivalence of such two kinds of groups (i.e., Theorem~\ref{g-f-g}). Furthermore, we introduce some subgroups of orbit braid group and study various possible relations among orbit braid group and its subgroups. This leads to the proof of
  Theorem~\ref{main}. In Section~\ref{calculation} we calculate the  orbit braid groups in $\mathbb{C}\times I$ with two typical actions on $\mathbb{C}$, from which we see that  the generalized braid group $Br(B_n)$ actually agrees with the orbit braid group $\mathcal{B}_n^{orb}(\mathbb{C}^{\times}, \mathbb{Z}_p)$ and $Br(D_n)$ is a subgroup of the orbit braid group $\mathcal{B}_n^{orb}(\mathbb{C}, \mathbb{Z}_p)$. In Section~\ref{extended fg} we define  the notion of  extended fundamental groups  in a general way in the category of topology, and give their some  characteristics  extracted from the discussions of orbit braids.
  Finally we review the generalized braid groups introduced by Brieskorn in Appendix~\ref{A}.

\section{Theory of orbit braids} \label{theory}

Given a topological group $G$ and a topological space $X$. Assume that $X$ admits an effective $G$-action. Then the orbit configuration space of the $G$-space $X$ is
defined by
$$
F_G(X,n)=\{(x_1,\ldots,x_n)\in X^{\times n} \ | \ G(x_i) \cap G(x_j)=\emptyset \textrm{ for } i\not=j\}
$$
with subspace topology, where $n\geq 2$ and $G(x)$ denotes the orbit of $x$. In the case where  $G$
acts trivially on $X$ or $G=\{e\}$, the space $F_G(X,n)$ is the classical
configuration space  $F(X,n)$.

\vskip .2cm
The action of $G$ on $X$ induces a natural action of $G^{\times n}$ on $F_G(X,n)$. In addition,
$F_G(X,n)$ also admits a canonical free action of the symmetric group $\Sigma_n$ on $F_G(X,n)$. However, generally  these two actions are not commutative.

\begin{remark}
The notion of orbit configuration space was introduced by Xicot\'encatl~\cite{X} in the thesis of his Ph.D. Since then,
the subject, with respect to the algebraic topology (especially cohomology) and relative topics of orbit configuration spaces,
has  been further developed.
\vskip .2cm
This equivariant case is quite different from the classical case. In particular, if the action of $G$ on $X$ is non-free, then the singular points (i.e., points of non-free orbit type) in $X$ will bring  difficulty to the study of problem. An effective approach to deal with this difficulty is to throw out all singular points from $X$ so as to further study (see, e.g., ~\cite{BG}). Another approach is to choose nice behaved equivariant manifolds. For example, in~\cite{CLW},  making use of two kinds of equivariant manifolds with non-free actions introduced by Davis and  Januszkiewicz~\cite{DJ} avoids  removing
 all singular points of orbit configuration spaces  since  the combinatorial structures of  the orbit spaces of the equivariant manifolds can determine all singular points, so that  an explicit formula of Euler characteristic for orbit
configuration spaces can be obtained in terms of combinatorics.
\end{remark}

\vskip .2cm
In the following, we shall  pay more attention on the case in which $X$ is a connected topological manifold $M$ of dimension greater than one, and $G$ is a finite  group. In this case $F_G(M,n)$ is connected. Here we shall focus on the study of orbit braids in $M\times I$ by associating to $F_G(M,n)$. We will see that our work does not only enrich the theory of braids, but also leads to a new understanding of how to use  paths. Actually, whichever paths are closed or unclosed,  by associating with the group action we can always form various kinds of  groups.  This extends the notion of fundamental groups.

\vskip .2cm
\subsection{Notions and properties of orbit braids}

A path $$\alpha=(\alpha_1, ..., \alpha_n): I\longrightarrow F_G(M,n)$$ uniquely determines a configuration $c(\alpha)=\{c(\alpha_1), ..., c(\alpha_n)\}$ of $n$ strings in $M\times I$, where $I=[0,1]$ admits a trivial action of $G$  and each string $c(\alpha_i)=\{(\alpha_i(s), s)|s\in I\}$ is homeomorphic to $I$. For each $s\in I$, since $\alpha(s)=(\alpha_1(s), ..., \alpha_n(s))\in F_G(M,n)$, it follows that the intersection of any two different $c(\alpha_i)$ and $c(\alpha_j)$ is empty, so we may write
$c(\alpha)=\coprod_{i=1}^nc(\alpha_i)$, which is naturally an unordered disjoint union of $n$ intervals.
Furthermore, it is easy to see that $c(\alpha)$ can  determine $n!$ paths $\alpha_{\sigma}=(\alpha_{\sigma(1)}, ..., \alpha_{\sigma(n)}), \sigma\in\Sigma_n$ in $F_G(M,n)$ such that $c(\alpha_{\sigma})=c(\alpha)$.

\vskip .2cm
 For the path $\alpha$ satisfying that $\alpha(0)=(x_1, ..., x_n)$ and $\alpha(1)=(x_{\sigma(1)}, ..., x_{\sigma(n)})$ for some $\sigma\in \Sigma_n$, if we forget the action of $G$ on $M$, then $c(\alpha)$ becomes a braid in the sense of Artin. Otherwise, $c(\alpha)$ would be different from the classical one. For  instance, see the following examples.

 \begin{example}\label{ex1}
 Consider the orbit configuration space $F_{\mathbb{Z}_2}(\mathbb{C},n)$ where the action of $\mathbb{Z}_2$ on $\mathbb{C}$ is given by $z\longmapsto -z$, so this action is non-free and  fixes only the origin  of $\mathbb{C}$. In the case of $n=2$, let us see two closed paths $\alpha, \beta: I\longrightarrow F_{\mathbb{Z}_2}(\mathbb{C},2)$ at the point ${\bf x}=(1,2)$ such that their corresponding  braids $c(\alpha)$ and $c(\beta)$ are as shown below:

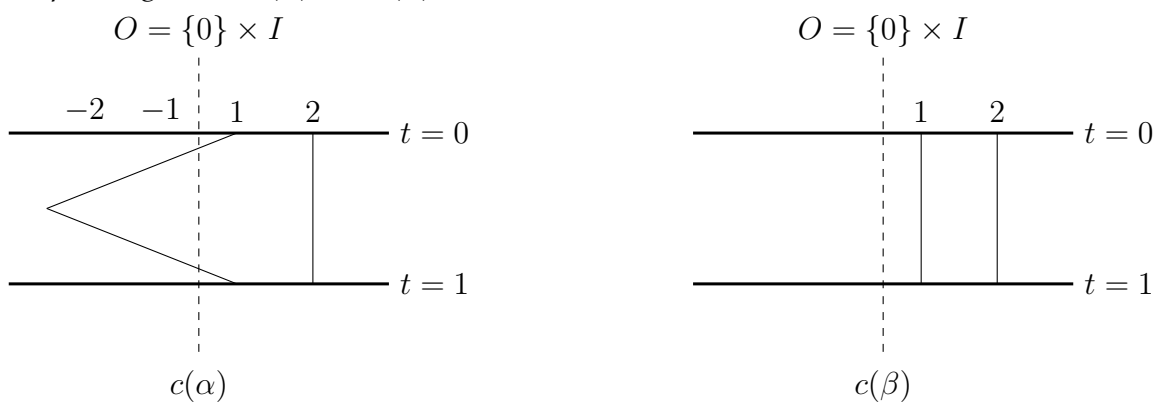
\begin{tikzpicture}
\draw [line width=0.04cm] (-7,0)--(-2,0);
\draw [line width=0.04cm] (2,0)--(7,0);
\node[right] at (-2,0) {$t=0$};
\node[right] at (7,0) {$t=0$};
\draw [line width=0.04cm] (-7,-2)--(-2,-2);
\draw [line width=0.04cm] (2,-2)--(7,-2);
\node[right] at (7,-2) {$t=1$};
\node[right] at (-2,-2) {$t=1$};
\draw [dashed] (-4.5,1)--(-4.5,-3);
\node
[above] at (-4.5,1) {$O=\{0\}\times I$};
\draw [dashed] (4.5,1)--(4.5,-3);
\node
[above] at (4.5,1) {$O=\{0\}\times I$};

\draw (-4,0)--(-6.5,-1);
\draw (-6.5,-1)--(-4,-2);
\node[above] at (-4,0) {$1$};

\draw (-3,0)--(-3,-2);
\node[above] at (-3,0) {$2$};
\node
[below] at (-4.5,-3) {$c(\alpha)$};

\draw (5,0)--(5,-2);
\node[above] at (5,0) {$1$};
\node[above] at (-5,0) {$-1$};
\draw (6,0)--(6,-2);
\node[above] at (6,0) {$2$};

\node[above] at (-6,0) {$-2$};
\node
[below] at (4.5,-3) {$c(\beta)$};
\end{tikzpicture}

\noindent If we forget the action of  $\mathbb{Z}_2$ on $\mathbb{C}$, then clearly $c(\alpha)$ and $c(\beta)$ are isotopic relative to endpoints in $\mathbb{C}\times I$. However, under the condition that $\mathbb{C}$ admits the action of  $\mathbb{Z}_2$, both $c(\alpha)$ and $c(\beta)$ are not isotopic since the first string of $c(\alpha)$ cannot go through the orbit of the second string of $c(\alpha)$, as we can see from the following left picture.

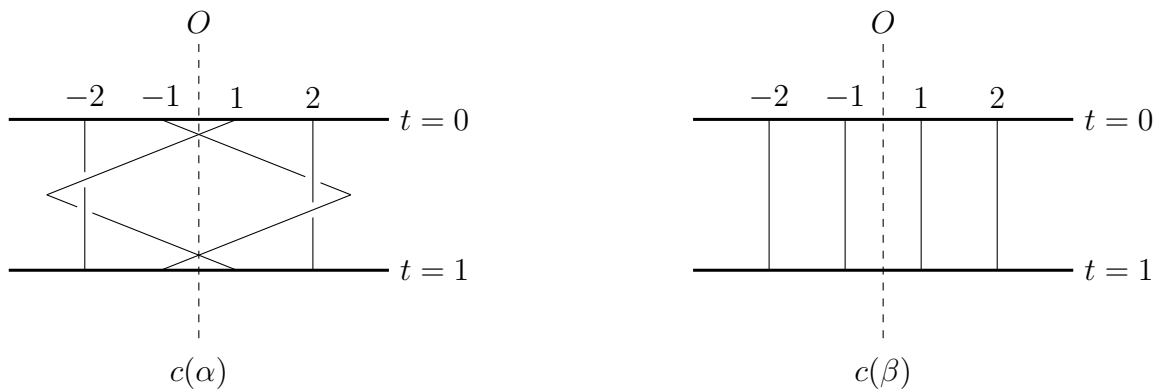
\begin{tikzpicture}
\draw [line width=0.04cm] (-7,0)--(-2,0);
\draw [line width=0.04cm] (2,0)--(7,0);
\node[right] at (-2,0) {$t=0$};
\node[right] at (7,0) {$t=0$};
\draw [line width=0.04cm] (-7,-2)--(-2,-2);
\draw [line width=0.04cm] (2,-2)--(7,-2);
\node[right] at (7,-2) {$t=1$};
\node[right] at (-2,-2) {$t=1$};
\draw [dashed] (-4.5,1)--(-4.5,-3);
\node
[above] at (-4.5,1) {$O$};
\draw [dashed] (4.5,1)--(4.5,-3);
\node
[above] at (4.5,1) {$O$};

\draw (-4,0)--(-6.5,-1);
\draw (-6.5,-1)--(-6.1,-1.16);\draw (-5.9,-1.24)--(-4,-2);
\node[above] at (-4,0) {$1$};

\draw (-3,0)--(-3, -1.1); \draw(-3, -1.3)--(-3,-2);
\node[above] at (-3,0) {$2$};

\draw (-5,0)--(-3.1, -0.76); \draw(-2.9, -0.84)--(-2.5,-1);
\draw (-2.5,-1)--(-5,-2);
\node[above] at (-5,0) {$-1$};

\draw (-6,0)--(-6, -0.7); \draw(-6, -0.9)--(-6,-2);
\node[above] at (-6,0) {$-2$};
\node [below] at (-4.5,-3) {$c(\alpha)$};

\draw (5,0)--(5,-2);
\node[above] at (5,0) {$1$};

\draw (6,0)--(6,-2);
\node[above] at (6,0) {$2$};

\draw (4,0)--(4,-2);
\node[above] at (4,0) {$-1$};

\draw (3,0)--(3,-2);
\node[above] at (3,0) {$-2$};
\node [below] at (4.5,-3) {$c(\beta)$};
\end{tikzpicture}
 \end{example}

\begin{remark} \label{h-i equiv}
In the theory of classical braids (cf. \cite{A2, Bir}), it is easy to see that
for  two paths $\alpha, \beta: I\longrightarrow F(M,n)$ with the same endpoints, $\alpha$ and  $\beta$ are homotopic relative to $\partial I$ (also write $\alpha\simeq\beta$ rel $\partial I$) if and only if  $c(\alpha)$ and $c(\beta)$ are isotopic relative to endpoints
  in $M\times I$, where $M$ is not equipped with any group action\footnote{Here the equivalence of $c(\alpha)$ and $c(\beta)$ up to isotopy is compatible with the Definition 3 of Artin's paper~\cite{A2} since $c(\alpha)$ and $c(\beta)$ are given by two paths in $F(M,n)$.}.
\end{remark}

\vskip .2cm
Since we are working  in the case of $M$ with an effective $G$-action, naturally we wish to know whether  the equivalence of homotopy and isotopy in Remark~\ref{h-i equiv} still holds in our case.

\begin{definition}\label{isotopy}
   Let $\alpha, \beta: I\longrightarrow F_G(M,n)$ be two paths with the same endpoints. We say that $c(\alpha)$ and $c(\beta)$ are {\bf isotopic with respect to the $G$-action} relative to endpoints in $M\times I$, denoted by
   $c(\alpha)\sim_{iso}^G c(\beta)$,  if
there exist $n$ homotopy maps $ \widehat{F}_i: I\times I \longrightarrow M\times I$ given by $\widehat{F}_i(s,t)=(F_i(s,t),s)$, $i=1, ..., n$, such that
\begin{itemize}
 \item[$(1)$] $\coprod_{i=1}^n \widehat{F}_i(s,0)=c(\alpha)$ and  $\coprod_{i=1}^n\widehat{F}_i(s,1)=c(\beta)$;
 \item[$(2)$] $\coprod_{i=1}^n \widehat{F}_i(0,t)=c(\alpha)|_{s=0}=c(\beta)|_{s=0}$ and
   $\coprod_{i=1}^n \widehat{F}_i(1,t)=c(\alpha)|_{s=1}=c(\beta)|_{s=1}$;
 \item[$(3)$] For any $(s, t)\in I\times I$, if $i\neq j$ then $G(F_i(s,t))\cap G(F_j(s,t))=\emptyset$.
 \end{itemize}
\end{definition}

With this understanding, we have the following result.
\begin{proposition}\label{first ns}
Let $\alpha, \beta: I\longrightarrow F_G(M,n)$ be two paths with the same endpoints. Then
$\alpha\simeq\beta$ rel $\partial I$ if and only if $c(\alpha)\sim_{iso}^G c(\beta)$.
\end{proposition}
\begin{proof}
Assume that $F=(F_1, ..., F_n): I\times I\longrightarrow F_G(M,n)$
is a homotopy relative to $\partial I$ from
 $\alpha$  to $\beta$. Then we can use $F$ to define $n$ homotopy maps
 $$\widehat{F}_i: I\times I\longrightarrow M\times I$$
 by $\widehat{F}_i(s,t)=(F_i(s,t),s)$, $i=1, ..., n$, satisfying (1)--(3) of Definition~\ref{isotopy}.
 Thus, $c(\alpha)\sim_{iso}^G c(\beta)$.

 \vskip .1cm

Conversely, suppose that $c(\alpha)\sim_{iso}^G c(\beta)$. Then  there are $n$ homotopy maps
 $$\widehat{F}_i: I\times I\longrightarrow M\times I$$
 by $\widehat{F}_i(s,t)=(F_i(s,t),s)$, $i=1, ..., n$, satisfying (1)--(3) of Definition~\ref{isotopy}.
Furthermore, these $F_i$ give a map $F=(F_1, ..., F_n):  I\times I\longrightarrow F_G(M,n)$,
  which is just the homotopy relative to $\partial I$ from $\alpha$ to $\beta$.
\end{proof}

For a path $\alpha=(\alpha_1, ..., \alpha_n): I\longrightarrow F_G(M,n)$, since $M$ admits an action of $G$, we may define the orbit of $\alpha$ as follows:
$$G^{\times n}(\alpha)=\{g\alpha=(g_1\alpha_1, ..., g_n\alpha_n)| g=(g_1, ..., g_n)\in G^{\times n}\},$$  a collection of $|G|^n$ paths in $F_G(M,n)$.    Then the corresponding configuration $c(\alpha)=\{c(\alpha_1), ..., c(\alpha_n)\}$ in $M\times I$ with trivial action of $G$ on $I$ gives its orbit configuration
$$\widetilde{c(\alpha)}=\{\widetilde{c(\alpha_1)}, ..., \widetilde{c(\alpha_n)}\}$$ where each orbit string $\widetilde{c(\alpha_i)}=\{hc(\alpha_i)|h\in G\}$ is the orbit of the string $c(\alpha_i)$ under the action of $G$, consisting of $|G|$ strings.  We note that the $|G|$ strings in each orbit string $\widetilde{c(\alpha_i)}$ may intersect to each other, but the intersection of any two different orbit strings $\widetilde{c(\alpha_i)}$ and  $\widetilde{c(\alpha_j)}$ must be empty since
$g\alpha(s)=(g_1\alpha_1(s), ..., g_n\alpha_n(s))\in F_G(M,n)$ for any $g=(g_1, ..., g_n)\in G^{\times n}$. Furthermore,
$\widetilde{c(\alpha)}$ can be written as $\coprod_{i=1}^n\widetilde{c(\alpha_i)}$, an unordered disjoint union,
 and it can  determine $|G|^nn!$ paths $g\alpha_{\sigma}=(g_1\alpha_{\sigma(1)}, ..., g_n\alpha_{\sigma(n)}), (g,\sigma)\in G^{\times n}\times\Sigma_n$, in $F_G(M,n)$ such that $\widetilde{c(g\alpha_{\sigma})}=\widetilde{c(\alpha)}$.

\vskip .2cm
We note that since the action of $G$ on $I$ is trivial, for each $1\leq i\leq n$, $hc(\alpha_i)=c(h\alpha_i), h\in G$. Thus,
for any $g\in G^{\times n}$, $\widetilde{c(g\alpha)}=\widetilde{c(\alpha)}$. Also, for two paths $\alpha$ and $\alpha'$, if
$\widetilde{c(\alpha)}=\widetilde{c(\alpha')}$, then there must be $g\in G^{\times n}$ and $\sigma\in \Sigma_n$ such that
$\alpha=g\alpha'_{\sigma}$.

\vskip .2cm

Now we are going to give the definition of orbit braids.
Choose a point ${\bf x}=(x_1, ..., x_n)$ in $F_G(M,n)$ such that for each $1\leq i\leq n$, the orbit $G(x_i)$ is of free type.
{\em Throughout the following, fix this point ${\bf x}$ as a base point}.

\vskip .2cm
Given a $\sigma\in \Sigma_n$, by ${\bf x}_{\sigma}$ we denote $(x_{\sigma(1)}, ..., x_{\sigma(n)})$. So, for $g=(g_1, ..., g_n)\in
G^{\times n}$, $g{\bf x}_{\sigma}=(g_1x_{\sigma(1)}, ..., g_nx_{\sigma(n)})$.

\begin{definition}\label{orbit braid}
 Let $\alpha=(\alpha_1, ..., \alpha_n): I\longrightarrow F_G(M,n)$ be a path such that $\alpha(0)={\bf x}$ and
$\alpha(1)=g{\bf x}_{\sigma}$ for some $(g, \sigma)\in G^{\times n}\times \Sigma_n$.
Then  $\widetilde{c(\alpha)}$ is called an {\bf orbit braid} in $M\times I$ .
\end{definition}

Obviously, each orbit braid  $\widetilde{c(\alpha)}$ has the property that $\widetilde{c(\alpha)}|_{s=0}$ and $\widetilde{c(\alpha)}|_{s=1}$ are homeomorphic to
$$\widetilde{c({\bf x})}=\{G(x_1), ...., G(x_n)\},$$
 an unordered collection of the orbits of $n$ coordinates of ${\bf x}$ under the action of $G$. Namely, two endpoints of each orbit braid $\widetilde{c(\alpha)}$ are the same up to homeomorphism. Here we also call $\widetilde{c({\bf x})}$ the {\em (unordered) orbit base point}.
 \vskip .2cm
 In the theory of ordinary braids, isotopy is used as the equivalence relation among ordinary braids.
 However, equivariant isotopy is not sufficient enough to be used as the equivalence relation among orbit braids. We can see this from the following example.

 \begin{example} \label{equ-relation}
Let the action of $\mathbb{Z}_2$ on $\mathbb{C}$ be the same as that in Example~\ref{ex1}.
 Consider the orbit configuration space $F_{\mathbb{Z}_2}(\mathbb{C},n)$.  In the case of $n=2$, take two closed paths $\alpha, \beta: I\longrightarrow F_{\mathbb{Z}_2}(\mathbb{C},2)$ at the base point ${\bf x}=(1,2)$ such that their corresponding ordinary braids $c(\alpha)$ and $c(\beta)$ are shown as follows:

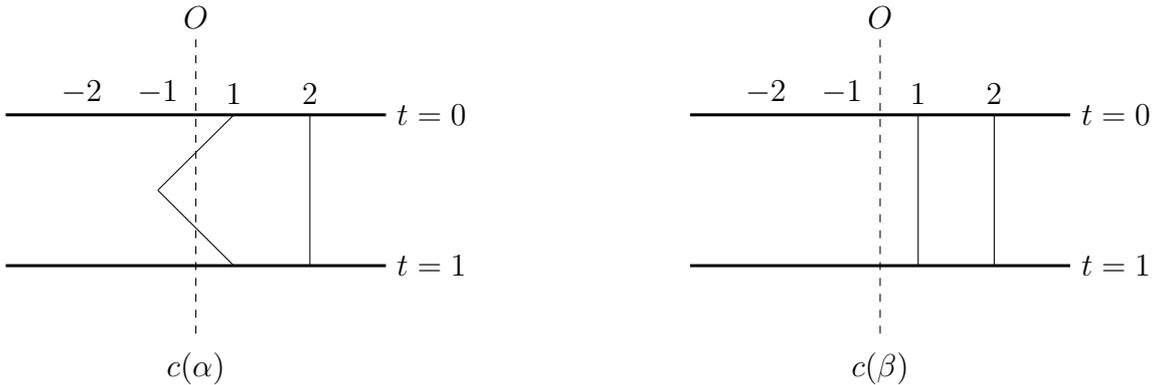
\begin{tikzpicture}
\draw [line width=0.04cm] (-7,0)--(-2,0);
\draw [line width=0.04cm] (2,0)--(7,0);
\node[right] at (-2,0) {$t=0$};
\node[right] at (7,0) {$t=0$};
\draw [line width=0.04cm] (-7,-2)--(-2,-2);
\draw [line width=0.04cm] (2,-2)--(7,-2);
\node[right] at (7,-2) {$t=1$};
\node[right] at (-2,-2) {$t=1$};
\draw [dashed] (-4.5,1)--(-4.5,-3);
\node
 [above] at (-4.5,1) {$O$};
 \draw [dashed] (4.5,1)--(4.5,-3);
 \node
  [above] at (4.5,1) {$O$};

\draw (-4,0)--(-5,-1);
\draw (-5,-1)--(-4,-2);
\node[above] at (-4,0) {$1$};
\node[above] at (-5,0) {$-1$};
\draw (-3,0)--(-3,-2);
\node[above] at (-3,0) {$2$};
\node[above] at (-6,0) {$-2$};
\node
 [below] at (-4.5,-3) {$c(\alpha)$};

\draw (5,0)--(5,-2);
\node[above] at (5,0) {$1$};
\node[above] at (4,0) {$-1$};

\draw (6,0)--(6,-2);
\node[above] at (6,0) {$2$};
\node[above] at (3,0) {$-2$};
\node[below] at (4.5,-3) {$c(\beta)$};
\end{tikzpicture}

\noindent Clearly, $c(\alpha)\sim_{iso}^G c(\beta)$. This means that orbit braids $\widetilde{c(\alpha)}$ and $\widetilde{c(\beta)}$ as shown below  are essentially the same in such a sense that the first string of $c(\alpha)$ can be deformed into the first string of $c(\beta)$ in $M\times I$ under the action of $G$.  However, $\widetilde{c(\alpha)}$ and $\widetilde{c(\beta)}$ are not equivariant isotopic since they are even not homeomorphic.

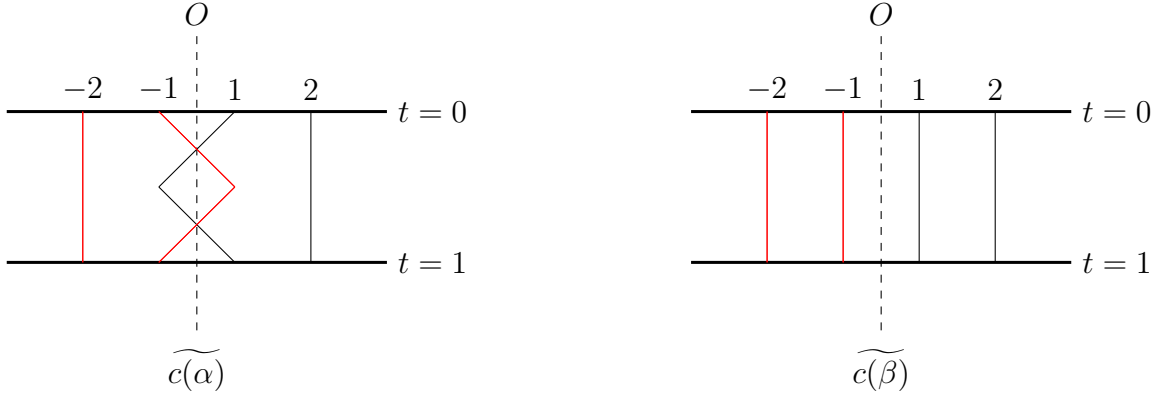
\begin{tikzpicture}
\draw [line width=0.04cm] (-7,0)--(-2,0);
\draw [line width=0.04cm] (2,0)--(7,0);
\node[right] at (-2,0) {$t=0$};
\node[right] at (7,0) {$t=0$};
\draw [line width=0.04cm] (-7,-2)--(-2,-2);
\draw [line width=0.04cm] (2,-2)--(7,-2);
\node[right] at (7,-2) {$t=1$};
\node[right] at (-2,-2) {$t=1$};
\draw [dashed] (-4.5,1)--(-4.5,-3);
\node[above] at (-4.5,1) {$O$};
 \draw [dashed] (4.5,1)--(4.5,-3);
 \node
  [above] at (4.5,1) {$O$};

\draw (-4,0)--(-5,-1);
\draw (-5,-1)--(-4,-2);
\node[above] at (-4,0) {$1$};

\draw (-3,0)--(-3,-2);
\node[above] at (-3,0) {$2$};

\filldraw[red] (-5,0)--(-4,-1);
\filldraw[red] (-4,-1)--(-5,-2);
\node[above] at (-5,0) {$-1$};

\filldraw[red] (-6,0)--(-6,-2);
\node[above] at (-6,0) {$-2$};
\node [below] at (-4.5,-3) {$\widetilde{c(\alpha)}$};

\draw (5,0)--(5,-2);
\node[above] at (5,0) {$1$};

\draw (6,0)--(6,-2);
\node[above] at (6,0) {$2$};

\filldraw[red] (4,0)--(4,-2);
\node[above] at (4,0) {$-1$};

\filldraw[red] (3,0)--(3,-2);
\node[above] at (3,0) {$-2$};
\node [below] at (4.5,-3) {$\widetilde{c(\beta)}$};
\end{tikzpicture}
\end{example}

Based upon this observation, we define the following equivalence relation among orbit braids, which is of nature to orbit braids.

\vskip .2cm
Let $\alpha$ and $\beta$ be two paths in $F_G(M, n)$ such that $\widetilde{c(\alpha)}$ and $\widetilde{c(\beta)}$  are two orbit braids at $\widetilde{c({\bf x})}$ in $M\times I$.

\begin{definition}\label{D1}
  We say that
 $\widetilde{c(\alpha)}$ and $\widetilde{c(\beta)}$  are {\bf equivalent},  denoted by  $\widetilde{c(\alpha)}\sim \widetilde{c(\beta)}$, if there are some $g$ and $h$ in $G^{\times n}$  such that
  $c(g\alpha)\sim_{iso}^G c(h\beta)$.
\end{definition}

\begin{remark}
It should be pointed out that if the action of $G$ on $M$ is free, then $\widetilde{c(\alpha)}\sim \widetilde{c(\beta)}$
if and only if $\widetilde{c(\alpha)}$ and  $\widetilde{c(\beta)}$ are equivariantly isotopic relative to endpoints of $\widetilde{c(\alpha)}$ and $\widetilde{c(\beta)}$. However, if  the action of $G$ on $M$ is not free, then generally $\widetilde{c(\alpha)}$ and  $\widetilde{c(\beta)}$ are not equivariantly isotopic even if $\widetilde{c(\alpha)}\sim \widetilde{c(\beta)}$, as seen in Example~\ref{equ-relation}.
\end{remark}

\begin{proposition}\label{sn-condition}
 $\widetilde{c(\alpha)}\sim \widetilde{c(\beta)}$ if and only if
there are two paths $\alpha'$ and $\beta'$ with $\widetilde{c(\alpha')}=\widetilde{c(\alpha)}$ and $\widetilde{c(\beta')}=\widetilde{c(\beta)}$,  such that $\alpha'$ is homotopic to $\beta'$ relative to $\partial I$.
\end{proposition}

\begin{proof}
This is a consequence of Proposition~\ref{first ns} and Definition~\ref{D1}.
\end{proof}

\begin{remark}
Assume that $\widetilde{c(\alpha)}\sim \widetilde{c(\beta)}$. Take two paths $\alpha'$ and $\beta'$ such that $\widetilde{c(\alpha')}=\widetilde{c(\alpha)}$ and $\widetilde{c(\beta')}=\widetilde{c(\beta)}$.  An easy argument shows that
$\alpha'$ is  homotopic to $\beta'$ relative to $\partial I=\{0,1\}$ if and only if  $\alpha'$ and $\beta'$
have the same endpoints. This is because the base point ${\bf x}$ possesses the property that the orbit
$G^{\times n}({\bf x})$ is of free type.
\end{remark}

Next, using the equivalence relation in Definition~\ref{D1}, define $\mathcal{B}_n^{orb}(M, G)$ as the set consisting of the equivalence classes of all orbit braids at the orbit base point $\widetilde{c({\bf x})}$ in $M\times I$.

\vskip .2cm
Take a class $[\widetilde{c(\alpha)}]$  in $\mathcal{B}_n^{orb}(M, G)$, for any two representatives $R'$ and $R''$ of $[\widetilde{c(\alpha)}]$,  Proposition~\ref{sn-condition} tells us that there exist two paths $\alpha',\alpha'': I\longrightarrow
F_G(M,n)$, which are homotopic relative to $\partial I$, such that $R'=\widetilde{c(\alpha'})$ and $R''=\widetilde{c(\alpha''})$.
We see easily that for any $g\in G^{\times n}$, $g\alpha'$ is still homotopic to $g\alpha''$ relative to $\partial I$, and
$R'=\widetilde{c(g\alpha'})$ and $R''=\widetilde{c(g\alpha''})$. So,
$$[\widetilde{c(\alpha)}]=[\widetilde{c(\alpha'})]=[\widetilde{c(\alpha''})]=[\widetilde{c(g\alpha'})]=
[\widetilde{c(g\alpha''})].$$
In addition, for any $\sigma\in \Sigma_n$, we have known that $\widetilde{c(\alpha)}=\widetilde{c(\alpha_\sigma)}$.
Thus this gives the following lemma.

\begin{lemma}\label{h-class}
For any $g\in G^{\times n}$ and any $\sigma\in \Sigma_n$, each class $[\widetilde{c(\alpha)}]$  in $\mathcal{B}_n^{orb}(M, G)$ is completely determined by the homotopy class  of  $g\alpha_\sigma$ relative to $\partial I$.
\end{lemma}

\begin{remark}\label{zero}
By Lemma~\ref{h-class}, we see that the homotopy class of $\alpha$ relative to $\partial I$ also determines the class $[\widetilde{c(\alpha)}]$. Without  loss of generality we may  assume that $\alpha(0)={\bf x}$. In fact, if $\alpha(0)\not={\bf x}$, then we may  write $\alpha(0)=h{\bf x}_\tau$ where $h\in G^{\times n}$ and $\tau\in \Sigma_n$. Moreover,
we get a path $\alpha'=h_{\tau^{-1}}^{-1}\alpha_{\tau^{-1}}$ such that
$\alpha'(0)={\bf x}$, and in particular, $\widetilde{c(\alpha'})=\widetilde{c(\alpha)}$, so the homotopy class  of $\alpha'$ relative to $\partial I$ completely determines the class $[\widetilde{c(\alpha)}]$.
\end{remark}

\begin{lemma}\label{unique}
Each class $[\widetilde{c(\alpha)}]$  in $\mathcal{B}_n^{orb}(M, G)$ determines a unique pair $(g, \sigma)\in G^{\times n}\times \Sigma_n$.
\end{lemma}

\begin{proof}
By  Remark~\ref{zero},  we may write $\alpha(0)={\bf x}=(x_1, ..., x_n)$. Next let us look at the ending point
$\alpha(1)$ of $\alpha$.  There must be a $g\in G^{\times n}$ and a
permutation $\sigma\in \Sigma_n$ such that  $\alpha(1)=g{\bf x}_{\sigma}$.
Consider a path $\alpha'$ such that $\widetilde{c(\alpha')}\sim\widetilde{c(\alpha)}$. Then there exists
 a pair $(h, \tau)\in G^{\times n}\times \Sigma_n$ such that $\alpha'$ is homotopic to $h\alpha_\tau$ relative to $\partial I$.
 So $\alpha'(0)=h\alpha_\tau(0)=h{\bf x}_\tau$ and $\alpha'(1)=h\alpha_\tau(1)=h(g{\bf x}_\sigma)_\tau$.
 We can use  $h$ and $\tau$ to change the endpoints of $\alpha'$ such that  $$h_{\tau^{-1}}^{-1}\alpha'_{\tau^{-1}}(0)=h_{\tau^{-1}}^{-1}(h{\bf x}_\tau)_{\tau^{-1}}={\bf x}$$
 and $h^{-1}_{\tau^{-1}}\alpha'_{\tau^{-1}}(1)=g{\bf x}_\sigma$.
Since $\widetilde{c(\alpha')}=\widetilde{c(h^{-1}_{\tau^{-1}}\alpha'_{\tau^{-1}})}$, we obtain that $\widetilde{c(\alpha')}$  also determines
the pair $(g, \sigma)$, implying that $(g, \sigma)$ does not depend upon  the choice of  representatives of $[\widetilde{c(\alpha)}]$.
\end{proof}

\begin{remark}
For the  unique pair $(g, \sigma)$ determined by $[\widetilde{c(\alpha)}]$ in Lemma~\ref{unique}, we see that as long as  $\alpha(0)={\bf x}$,  then the ending point $\alpha(1)=g{\bf x}_{\sigma}$ completely determines this pair. Furthermore,
  up to the order of ${\bf x}$,
the ending point $\alpha(1)=g{\bf x}_{\sigma}$ can be changed into
$$(g_{\sigma^{-1}(1)}x_1, ..., g_{\sigma^{-1}(n)}x_n),$$
which just gives an action of $\sigma$ on $g$.
\end{remark}

Let $\pi_1^E(F_G(M,n), {\bf x}, {\bf x}^{orb})$ be the set consisting of the homotopy classes relative to $\partial I$ of all paths
$\alpha: I\longrightarrow F_G(M,n)$ with $\alpha(0)={\bf x}$ and $\alpha(1)\in {\bf x}^{orb}$, where
${\bf x}^{orb}=\{g{\bf x}_\sigma| g\in G^{\times n}, \sigma\in \Sigma_n\}$, which is the orbit set at ${\bf x}$ under two actions of $G^{\times n}$ and $\Sigma_n$.
With Proposition~\ref{sn-condition}, Lemma~\ref{h-class} and Lemma~\ref{unique} together, we have

\begin{corollary}\label{one-one}
$\mathcal{B}_n^{orb}(M, G)$  bijectively corresponds to  $\pi_1^E(F_G(M,n), {\bf x},  {\bf x}^{orb})$ as sets.
\end{corollary}

\begin{remark}\label{permutation}
Given $\sigma\in \Sigma_n$ and $g\in G^{\times n}$,  we see easily  that  $\pi_1^E(F_G(M,n), {\bf x},  {\bf x}^{orb})$
  bijectively corresponds to $\pi_1^E(F_G(M,n), g{\bf x}_{\sigma},  {\bf x}^{orb})$ by mapping $[\alpha]$ to
  $[g\alpha_{\sigma}]$, so
$\pi_1^E(F_G(M,n), g{\bf x}_{\sigma},  {\bf x}^{orb})$ also  bijectively corresponds to $\mathcal{B}_n^{orb}(M, G)$. Of course,  $[\widetilde{c(\alpha)}]=[\widetilde{c(g\alpha_{\sigma}})]$.
\end{remark}

With the above understanding, for each class $[\widetilde{c(\alpha)}]$ in $\mathcal{B}_n^{orb}(M, G)$, {\em we will always assume that the path $\alpha$ is in a class of $\pi_1^E(F_G(M,n), {\bf x},  {\bf x}^{orb})$}. By $[\alpha]$ we denote the homotopy class (relative to $\partial I$) determined by $\alpha$.

\subsection{Groups of orbit braids and their homotopy descriptions}
 Let $[\widetilde{c(\alpha)}]$ and $[\widetilde{c(\beta)}]$ be two classes in $\mathcal{B}_n^{orb}(M, G)$.
 \vskip .2cm
  First let us consider the  operation between  $\widetilde{c(\alpha)}$ and
$\widetilde{c(\beta)}$ in an intuitive way.  Since two orbit braids have the same endpoints, intuitively we can obtain a new orbit braid $\widetilde{c(\alpha)}\circ\widetilde{c(\beta)}$
  by gluing the starting points of $n$ orbit strings in $\widetilde{c(\beta)}$ to the ending points of $n$ orbit strings in $\widetilde{c(\alpha)}$. More precisely,
 $$\widetilde{c(\alpha)}\circ\widetilde{c(\beta)}|_{s\in I}=\begin{cases}
 \widetilde{c(\alpha)}|_{2s\in I} & \text{if } s\in [0, {1\over 2}]\\
 \widetilde{c(\beta)}|_{2s-1\in I} & \text{if } s\in [{1\over 2}, 1].
 \end{cases}
 $$
 Clearly this operation $\circ$ is well-defined, but it is not associative as in the case of the operation among ordinary paths.
  By Corollary~\ref{one-one}, this  new orbit braid $\widetilde{c(\alpha)}\circ\widetilde{c(\beta)}$ should be  determined by a path $\gamma: I\longrightarrow F_G(M,n)$ with $\gamma(0)={\bf x}$ and $\gamma(1)\in {\bf x}^{orb}$. Such a path $\gamma$ can be constructed as follows:

  \vskip .2cm

  By Lemma~\ref{unique}, there exist a unique pair $(g, \sigma)\in G^{\times n}\times \Sigma_n$ such that  $\alpha(1)=g{\bf x}_\sigma$.  Of course,  $[\widetilde{c(\beta)}]$ also determines a unique pair $(h, \tau)\in G^{\times n}\times\Sigma_n$ such that  $\beta(1)=h{\bf x}_\tau$.
 Consider $\widehat{\beta}=g\beta_{\sigma}$, since $\beta(0)={\bf x}$, we have that
 $\widehat{\beta}(0)=\alpha(1)=g{\bf x}_\sigma$, and we know from Remark~\ref{permutation} that
 $\widetilde{c(\beta)}= \widetilde{c(\widehat{\beta}})$. Then we can construct a new path
 $$\gamma(s)=\alpha\circ\widehat{\beta}(s)=\begin{cases}
\alpha(2s) & \text{if } s\in [0, {1\over 2}]\\
\widehat{\beta}(2s-1) & \text{if } s\in [{1\over 2}, 1]
\end{cases}$$
with $\gamma(0)={\bf x}$ and $\gamma(1)=g\beta_{\sigma}(1)=g(h{\bf x}_\tau)_\sigma=
gh_\sigma{\bf x}_{\sigma\tau}$,  as desired.

 \begin{remark}\label{pair operation}
 In the above construction of $\gamma$, we see that
 two pairs  $(g, \sigma)$ and $(h, \tau)$ actually produce a new  pair $(gh_\sigma, \sigma\tau)$, which is uniquely determined by
$[\widetilde{c(\alpha)}\circ\widetilde{c(\beta)}]=[\widetilde{c(\gamma)}]$.
\end{remark}

Now we define an operation $*$ on $\mathcal{B}_n^{orb}(M, G)$ by
$$[\widetilde{c(\alpha)}]*[\widetilde{c(\beta)}]=[\widetilde{c(\alpha)}\circ\widetilde{c(\beta)}].$$
 We claim that the operation $*$ is well-defined and associative. By
 Corollary~\ref{one-one}, it suffices to show that for any $\alpha'\in [\alpha]$ and any $\beta'\in [\beta]$,
$$[\widetilde{c(\alpha'})]*[\widetilde{c(\beta'})]=[\widetilde{c(\alpha'})\circ\widetilde{c(\beta'})]
=[\widetilde{c(\alpha)}\circ\widetilde{c(\beta)}]=[\widetilde{c(\alpha)}]*[\widetilde{c(\beta)}].$$
Since $\alpha(i)=\alpha'(i)$ and $\beta(i)=\beta'(i)$ for $i=0,1$, we have that $g\beta'_{\sigma}(0)=\alpha(1)$ and  $\widetilde{c(\beta)}= \widetilde{c(g\beta'_{\sigma}})$.
In a similar way to the definition of $\gamma$ as above,
we may define $\gamma'=\alpha'\circ(g\beta'_\sigma)$.
Furthermore,  homotopy theory \cite{S} tells us that
$\gamma'=\alpha'\circ(g\beta'_{\sigma})$ is homotopic to $\gamma=\alpha\circ(g\beta_{\sigma})$ relative to $\partial I$, implying that the operation $*$ is well-defined. Since the operation $*$ is essentially reduced to the operation on the homotopy classes of  paths, it is also associative.

\begin{proposition}\label{group}
$\mathcal{B}_n^{orb}(M, G)$ forms a group under the operation $*$, called the {\em \bf orbit braid group} of the $G$-manifold $M$.
\end{proposition}

\begin{proof}
Obviously, the class $[\widetilde{c(c_{\bf x}})]$ is just the unit element, where $c_{\bf x}$ is the constant path with $c_{\bf x}(s)={\bf x}, s\in I$. Let $[\widetilde{c(\alpha)}]$ be an element in $\mathcal{B}_n^{orb}(M, G)$.
Consider the inverse path $\overline{\alpha}$ of $\alpha$, i.e., $\overline{\alpha}(s)=\alpha(1-s)$. It is well-known in homotopy theory that $\alpha\circ \overline{\alpha}$ is homotopic to $c_{\bf x}$. Thus,
$$[\widetilde{c(\alpha)}]*[\widetilde{c(\overline{\alpha})}]=[\widetilde{c(\alpha)}\circ\widetilde{c(\overline{\alpha})}]=
[\widetilde{c(c_{\bf x})}]$$
gives that $[\widetilde{c(\alpha)}]^{-1}=[\widetilde{c(\overline{\alpha})}]$.
\end{proof}

When the action of $G$ on $M$ is trivial or the group $G$ is just the trivial group $\{e\}$, clearly $\mathcal{B}_n^{orb}(M,G)$ will degenerate into the ordinary braid group $B_n(M)$. Thus the notion of orbit braid  group is a generalization for ordinary braid groups.

\vskip .2cm

Putting some restrictions on endpoints of orbit braids, we may define some subgroups of $\mathcal{B}_n^{orb}(M, G)$ as follows.

\begin{definition} \label{subgroup-def} {\rm (Subgroups of  $\mathcal{B}_n^{orb}(M, G)$)}
\begin{itemize}
 \item[$(1)$] Those classes $[\widetilde{c(\alpha)}]$ with $\alpha(1)\in G^{\times n}({\bf x})$ of $\mathcal{B}_n^{orb}(M, G)$ form a subgroup of
 $\mathcal{B}_n^{orb}(M, G)$, which is called the {\bf pure orbit braid group}, denoted by $\mathcal{P}_n^{orb}(M,G)$.
 \item[$(2)$] Those classes $[\widetilde{c(\alpha)}]$ with $\alpha(1)\in \Sigma_n({\bf x})=\{{\bf x}_{\sigma}|\sigma\in \Sigma_n\}$ of $\mathcal{B}_n^{orb}(M, G)$ form a subgroup of
 $\mathcal{B}_n^{orb}(M, G)$, which is  called the {\bf  braid group}, denoted by $\mathcal{B}_n(M,G)$.
 \item[$(3)$] Those classes $[\widetilde{c(\alpha)}]$ with $\alpha(1)={\bf x}$ of $\mathcal{B}_n^{orb}(M, G)$ form a subgroup of
 $\mathcal{B}_n^{orb}(M, G)$, which is  called the {\bf pure braid group}, denoted by $\mathcal{P}_n(M,G)$.
\end{itemize}
\end{definition}

The above argument gives an insight to the $\pi_1^E(F_G(M,n), {\bf x}, {\bf x}^{orb})$ such that it can also form a group.  Actually we can  endow an operation $\bullet$ on $\pi_1^E(F_G(M,n), {\bf x}, {\bf x}^{orb})$ defined by
\begin{equation}\label{op}
[\alpha]\bullet[\beta]=[\alpha\circ(g\beta_{\sigma})]
\end{equation}
where  $(g, \sigma) \in G^{\times n}\times \Sigma_n$ is the unique pair determined by $[\widetilde{c(\alpha}]$. Then it is easy to see that  $\pi_1^E(F_G(M,n), {\bf x}, {\bf x}^{orb})$ becomes a group under this operation, and it is called the {\em extended fundamental group} of $F_G(M,n)$ at ${\bf x}^{orb}$. Thus we have

 \begin{theorem}\label{group'}
$\pi_1^E(F_G(M,n), {\bf x}, {\bf x}^{orb})$ forms a group under the operation $\bullet$. Furthermore,
the map $$\Lambda:\pi_1^E(F_G(M,n), {\bf x}, {\bf x}^{orb})\longrightarrow\mathcal{B}_n^{orb}(M, G)$$ given by $[\alpha]\longmapsto[\widetilde{c(\alpha}]$ is an isomorphism.
\end{theorem}

Obviously,  $\mathcal{P}_n^{orb}(M,G)$ bijectively corresponds to the $\pi_1^E(F_G(M,n), {\bf x},  G^{\times n}({\bf x}))$ defined in a same way as above which is also  a group under the operation $\bullet$. So $\mathcal{P}_n^{orb}(M,G)\cong\pi_1^E(F_G(M,n), {\bf x},  G^{\times n}({\bf x}))$. Similarly, $\mathcal{B}_n(M,G)\cong\pi_1^E(F_G(M,n), {\bf x},  \Sigma_n({\bf x}))$
and $$\mathcal{P}_n(M,G)\cong \pi_1^E(F_G(M,n), {\bf x},   {\bf x})=\pi_1(F_G(M,n), {\bf x}).$$ Here $\pi_1^E(F_G(M,n), {\bf x},  G^{\times n}({\bf x}))$ and $\pi_1^E(F_G(M,n), {\bf x},  \Sigma_n({\bf x}))$ are also called the {\em extended fundamental groups} of
$F_G(M,n)$ at $G^{\times n}({\bf x})$ and  $\Sigma_n({\bf x})$, respectively.
Therefore, those subgroups of $\mathcal{B}_n^{orb}(M,G)$ can be described in terms of the homotopy classes of paths in $F_G(M,n)$.

\begin{corollary} \label{homotopy version}
Homotopy decriptions of  subgroups $\mathcal{P}_n^{orb}(M,G)$, $\mathcal{B}_n(M,G)$ and $\mathcal{P}_n(M,G)$.
\begin{itemize}
 \item[$(1)$]  $\mathcal{P}_n^{orb}(M,G)\cong \pi_1^E(F_G(M,n), {\bf x},  G^{\times n}({\bf x}))$;
 \item[$(2)$]  $\mathcal{B}_n(M,G)\cong \pi_1^E(F_G(M,n), {\bf x},  \Sigma_n({\bf x}))$;
 \item[$(3)$]  $\mathcal{P}_n(M,G)\cong \pi_1^E(F_G(M,n), {\bf x},   {\bf x})=\pi_1(F_G(M,n), {\bf x})$.
\end{itemize}
\end{corollary}

Corollary~\ref{homotopy version} tells us that the pure braid group $\mathcal{P}_n(M,G)$ can be realized as the
 fundamental group $\pi_1(F_G(M,n), {\bf x})$.
Later on, we shall show that  $\mathcal{B}_n(M,G)$ can  be realized as the fundamental group of $F_G(M,n)/\Sigma_n$ yet, and we shall see much more information on $\mathcal{B}_n^{orb}(M,G)$ and $\mathcal{P}_n^{orb}(M,G)$.

\begin{remark}
The above viewpoint can also be used in the theory of ordinary braids. Consider the case in which $G=\{e\}$.
 Then $\mathcal{B}_n^{orb}(M, G)$  degenerates into the ordinary braid group $B_n(M)$, which is isomorphic to the extended fundamental group $\pi_1^E(F(M,n), {\bf x}, \Sigma_n({\bf x}))$ of $F(M,n)$ at $\Sigma_n({\bf x})$. In this case, there is the following short exact sequence
$$1\longrightarrow \pi_1(F(M,n), {\bf x})\longrightarrow \pi_1^E(F(M,n), {\bf x}, \Sigma_n({\bf x}))\longrightarrow\Sigma_n\longrightarrow 1$$
from which we see that the extended fundamental group $\pi_1^E(F(M,n), {\bf x}, \Sigma_n({\bf x}))$ is actually the fundamental group of the unordered configuration space $F(M,n)/\Sigma_n$.  However, the  cases of $G\not=\{e\}$ will be  quite different.
\end{remark}

\subsection{Short exact sequences}

Let $\varphi:\Sigma_n\longrightarrow \Aut(G^{\times n})$ be a homomorphism defined by
$$\varphi(\sigma)(g)=g_\sigma=(g_{\sigma(1)}, ..., g_{\sigma(n)})$$
where $\sigma\in \Sigma_n$ and $g=(g_1, ..., g_n)\in G^{\times n}$. Then $\varphi$ gives a semidirect product
$G^{\times n}\rtimes_{\varphi}\Sigma_n$, where the operation $\cdot$ on $G^{\times n}\rtimes_{\varphi}\Sigma_n$ is given by
$$(g, \sigma)\cdot (h, \tau)=(gh_\sigma, \sigma\tau)$$
for $(g, \sigma), (h, \tau)\in G^{\times n}\rtimes_{\varphi}\Sigma_n$. Then, by Lemma~\ref{unique} and
Remark~\ref{pair operation}, we can define a homomorphism $$\Phi:  \mathcal{B}_n^{orb}(M, G)\longrightarrow G^{\times n}\rtimes_{\varphi}\Sigma_n$$ by
$\Phi([\widetilde{c(\alpha)}])=(g, \sigma)$, where $(g, \sigma)$ is the unique pair determined by $[\widetilde{c(\alpha)}]$.

\begin{lemma} \label{epi}
The homomorphism $$\Phi:  \mathcal{B}_n^{orb}(M, G)\longrightarrow G^{\times n}\rtimes_{\varphi}\Sigma_n$$ is an epimorphism.
\end{lemma}

\begin{proof}
Given a pair $(g, \sigma)$ in $G^{\times n}\rtimes_{\varphi}\Sigma_n$, since $F_G(M,n)$ is connected (so it also is path-connected), there must be a path
 $\alpha: I\longrightarrow F_G(M,n)$ such that $\alpha(0)={\bf x}$ and $\alpha(1)=g{\bf x}_{\sigma}$, which gives $\Phi([\widetilde{c(\alpha)}])=(g, \sigma)$. Furthermore, it follows from Corollary~\ref{one-one} or Theorem~\ref{group'} that $\Phi$ is an epimorphism.
\end{proof}

Based upon the Definition~\ref{subgroup-def}, when $\Phi$ is restricted to $\mathcal{P}_n^{orb}(M, G)$, each class $[\widetilde{c(\alpha)}]$ will uniquely determine the pair
$(g, e_{\Sigma_n})$ where $e_{\Sigma_n}$ is the unit element of $\Sigma_n$. Thus, $\Phi$ induces a homomorphism
$$\Phi_G: \mathcal{P}_n^{orb}(M, G)\longrightarrow G^{\times n}$$
given by $\Phi_G([\widetilde{c(\alpha)}])=g$, which is surjective.
When $\Phi$ is restricted to $\mathcal{B}_n(M,G)$,  each class $[\widetilde{c(\alpha)}]$ will uniquely determine the pair
$(e_{G^{\times n}}, \sigma)$ where $e_{G^{\times n}}$ is the unit element of $G^{\times n}$. So  $\Phi$ induces a homomorphism
 $$\Phi_\Sigma: \mathcal{B}_n(M,G)\longrightarrow\Sigma_n$$
 which  is also an epimorphism.

 \vskip .2cm
 An observation shows that each of kernels $\Ker\Phi$, $\Ker \Phi_G$ and $\Ker \Phi_\Sigma$ is just the pure braid group $\mathcal{P}_n(M,G)$.

 \vskip .2cm
 On the other hand, there are two  natural projections $p_\Sigma: G^{\times n}\rtimes_{\varphi}\Sigma_n\longrightarrow\Sigma_n$ and
 $p_G: G^{\times n}\rtimes_{\varphi}\Sigma_n\longrightarrow G^{\times n}$, which give two maps
  $$p_\Sigma\circ \Phi: \mathcal{B}_n^{orb}(M, G)\longrightarrow\Sigma_n$$
  and
  $$p_G\circ \Phi: \mathcal{B}_n^{orb}(M, G)\longrightarrow G^{\times n}.$$
We see by Lemma~\ref{epi} that such two maps are still surjective because $F_G(M,n)$ is path-connected. In addition, it is easy to see that   $\Ker p_\Sigma\circ \Phi=\mathcal{P}_n^{orb}(M,G)$, and $\Ker p_G\circ \Phi=\mathcal{B}_n(M,G)$. However, we note that $p_G\circ \Phi$
is not a group homomorphism since $p_G$ is not a group homomorphism, and  $p_\Sigma\circ \Phi$ is still a group homomorphism.

\vskip .2cm

Together with all arguments above, we have

\begin{theorem} \label{main result}
The following diagram commutes and contains five short exact sequences.
$$
\xymatrix{
1 \ar[dr] & & & & 1\\
 & \mathcal{P}_n^{orb}(M,G)    \ar[rr]^{\Phi_G}  \ar[dr] &   & G^{\times n} \ar[ur]  &    \\
1 \ar[r] & \mathcal{P}_n(M,G) \ar[u] \ar[d] \ar[r]  & \mathcal{B}_n^{orb}(M,G) \ar[ur] \ar[r]^{\Phi} \ar[dr] & G^{\times n}\rtimes_{\varphi}\Sigma_n
\ar[u]_{p_G} \ar[r] \ar[d]^{p_\Sigma} & 1 \\
& \mathcal{B}_n(M,G) \ar[rr]^{\Phi_\Sigma}  \ar[ur] & & \Sigma_n \ar[dr] &   \\
1 \ar[ur] &&&& 1
}
$$
\end{theorem}
\begin{remark} We note that because the map $\mathcal{B}_n^{orb}(M,G) \longrightarrow  G^{\times n}$ is not a group homomorphism,
$$1\longrightarrow \mathcal{B}_n(M,G) \longrightarrow \mathcal{B}_n^{orb}(M,G) \longrightarrow  G^{\times n}\longrightarrow 1$$
is not an exact sequence in the sense that all maps must be group homomorphisms.  However, it can still be regarded as  an exact sequence  in the sense of Switzer for topological spaces \cite{S}.
\end{remark}
\subsection{The action of $G^{\times n}\rtimes_\varphi \Sigma_n $ on  $F_G(M,n)$}

There is a canonical action of $G^{\times n}\rtimes_\varphi \Sigma_n $ on the orbit configuration space $F_G(M,n)$:
$$\Gamma:(G^{\times n}\rtimes_\varphi \Sigma_n) \times F_G(M,n) \longrightarrow F_G(M,n)$$
defined by $((g, \sigma), y)\longmapsto gy_\sigma$. Let $p: F_G(M,n)\longrightarrow F_G(M,n)/G^{\times n}\rtimes_\varphi \Sigma_n$ denote the orbit projection of this action.

\vskip .2cm
The action $\Gamma$ restricted to subgroups $G^{\times n}$ and $\Sigma_n$ gives the actions
$$\Gamma_G: G^{\times n}\times F_G(M,n)\longrightarrow F_G(M,n)$$
defined by $(g, y)\longmapsto gy$, and
$$\Gamma_\Sigma: \Sigma_n\times F_G(M,n)\longrightarrow F_G(M,n)$$
defined by $(\sigma, y)\longmapsto y_\sigma$, respectively.
Then we obtain two  orbit projections corresponding to actions $\Gamma_G$ and $\Gamma_\Sigma$:
$$p^G: F_G(M,n)\longrightarrow F_G(M,n)/G^{\times n}=F(M/G,n)$$
and
$$p^\Sigma: F_G(M,n)\longrightarrow F_G(M,n)/\Sigma_n.$$
There is the following commutative diagram with respect to the above orbit projections:
\begin{equation}\label{diagram}
\xymatrix{
                &          F_G(M,n) \ar[dl]_{p^G} \ar[dd]^{p} \ar[dr]^{p^\Sigma}  &  \\
  F_G(M,n)/G^{\times n} \ar[dr] &  & F_G(M,n)/\Sigma_n \ar[dl]\\
          &    F_G(M,n)/G^{\times n}\rtimes_\varphi \Sigma_n & \\
             }
\end{equation}

First let us look at the action $\Gamma_\Sigma$. It is easy to see that this action is free.

\begin{proposition}\label{bf}
The braid group $\mathcal{B}_n(M,G)$ is isomorphic to  $\pi_1(F_G(M,n)/\Sigma_n, p^\Sigma({\bf x}))$.
\end{proposition}
\begin{proof}
We see from Corollary~\ref{homotopy version} that $\mathcal{B}_n(M,G)\cong \pi_1^E(F_G(M,n), {\bf x},  \Sigma_n({\bf x}))$.
Then the projection $p^\Sigma$ induces the map $$p^\Sigma_*: \pi_1^E(F_G(M,n), {\bf x},  \Sigma_n({\bf x})) \longrightarrow \pi_1(F_G(M,n)/\Sigma_n, p^\Sigma({\bf x})),$$ which is an isomorphism by the theory of covering spaces.
\end{proof}

\begin{remark}
Corollary~\ref{homotopy version}, with Proposition~\ref{bf} together,  tell us that  the  short exact sequence in Theorem~\ref{main result}
 $$1\longrightarrow \mathcal{P}_n(M,G)\longrightarrow \mathcal{B}_n(M,G)\longrightarrow\Sigma_n\longrightarrow 1$$
  geometrically corresponds to the
 short exact sequence
$$1\longrightarrow \pi_1(F_G(M,n), {\bf x})\longrightarrow \pi_1(F_G(M,n)/\Sigma_n, p^\Sigma({\bf x}))\longrightarrow\Sigma_n\longrightarrow 1$$
given by the fibration $F_G(M,n)\longrightarrow F_G(M,n)/\Sigma_n$ with fiber $\Sigma_n$.
\end{remark}

\vskip .2cm
For two elements $((g_1,\dots,g_n), \sigma)$ and $((h_1,\dots,h_n),\tau)$ in
$G^{\times n}\rtimes_\varphi \Sigma_n $, if $$(g_1y_{\sigma(1)},\dots,g_ny_{\sigma(n)})= (h_1y_{\tau(1)},\dots,h_ny_{\tau(n)})$$
then it is easy to see that $\sigma=\tau$ and $g_i^{-1}h_i\in G_{y_{\sigma(i)}}, i=1, ..., n$, where $G_{y_{\sigma(i)}}$ is the isotropy subgroup at $y_{\sigma(i)}\in M$ of $G$. Thus,  in general, this action $\Gamma$ is non-free.
The following fact is obvious.
\begin{lemma}\label{free action}
The action of $G$ on $M$ is free if and only if the action $\Gamma$ is free.
\end{lemma}

 In a similar way to the proof of Proposition~\ref{bf}, making use of Theorem~\ref{main result} and Lemma~\ref{free action}
  via two homomorphisms $\Phi$ and $\Phi_G$ we see that

\begin{corollary} \label{free-1}
If the action of $G$ on $M$ is free, then $\mathcal{B}_n^{orb}(M,G)$ (resp. $\mathcal{P}_n^{orb}(M,G)$) is realizable as the fundamental group of the orbit space
$F_G(M,n)/G^{\times n}\rtimes_\varphi \Sigma_n$ (resp. $F_G(M,n)/G^{\times n}$), which is exactly $Br(M^{\times n}, G^{\times n}\rtimes_\varphi \Sigma_n)$ (resp. $Br(M^{\times n}, G^{\times n})$) in the sense of Vershinin.
\end{corollary}

\subsection{Liftings of paths}

For the projection $p^G: F_G(M,n)\longrightarrow F(M/G,n)$, write $p^G=(p^G_1, ..., p^G_n)$ and $\overline{\bf x}=p^G({\bf x})$. Given a $\sigma\in \Sigma_n$, consider the path
$\overline{\alpha}=(\overline{\alpha}_1, ..., \overline{\alpha}_n): I\longrightarrow F(M/G,n)$ from $\overline{\bf x}$ to $\overline{\bf x}_{\sigma}$, we then have
$$(p^G)^{-1}(\overline{\alpha}(I))=((p^G_1)^{-1}(\overline{\alpha}_1(I)), ..., (p^G_n)^{-1}(\overline{\alpha}_n(I))).$$
Since $G$ is finite, it is not difficult to see that there must be at least $|G|^n$ path liftings $\alpha: I\longrightarrow F_G(M,n)$
$$\xymatrix{
                &          F_G(M,n) \ar[d]^{p^G}     \\
  I \ar[ur]^{\alpha}  \ar[r]^{\overline{\alpha}} &    \ \ \ \ \ F(M/G, n)           }
$$
such that
$$G^{\times n}(\alpha(I))=(p^G)^{-1}(\overline{\alpha}(I))=((p^G_1)^{-1}(\overline{\alpha}_1(I)), ..., (p^G_n)^{-1}(\overline{\alpha}_n(I))).$$
In particular, it is not difficult to see that there must be path liftings $\alpha$ with $\alpha(0)={\bf x}$ of $\overline{\alpha}$, so
$$[\alpha]\in \pi_1^E(F_G(M,n), {\bf x}, G^{\times n}({\bf x}_\sigma))\subset \pi_1^E(F_G(M,n), {\bf x}, {\bf x}^{orb}).$$  Let $\pi_1^E(F(M/G, n), \overline{{\bf x}}, \overline{\bf x}_{\sigma})$ denote the set of homotopy classes relative to $\partial I$ of all paths $\overline{\alpha}$ with $\overline{\alpha}(0)=\overline{{\bf x}}$ and $\overline{\alpha}(1)= \overline{\bf x}_{\sigma}$. We can use the same way as in (\ref{op}) to give an operation $\star$ on $\pi_1^E(F(M/G, n), \overline{{\bf x}}, \overline{\bf x}_{\sigma})$ by
$$[\overline{\alpha}]\star[\overline{\beta}]=[\overline{\alpha}\circ \overline{\beta}_\sigma].$$
So $\pi_1^E(F(M/G, n), \overline{{\bf x}}, \overline{\bf x}_{\sigma})$ is also a group under such operation, still called the extended fundamental group.
Then we have that

\begin{lemma} \label{sur}
 The projection $p^G:F_G(M,n)\longrightarrow F(M/G,n)$ induces an epimorphism
$$p^G_*:  \pi_1^E(F_G(M,n), {\bf x}, G^{\times n}({\bf x}_\sigma))\longrightarrow \pi_1^E(F(M/G, n), \overline{{\bf x}}, \overline{\bf x}_{\sigma})$$
by $p^G_*([\alpha])=[p^G\alpha]$.
\end{lemma}

\begin{remark}
Set
$$\mathcal{S}((p^G)^{-1}\overline{\alpha})=\coprod_{i=1}^n \{((p^G)^{-1}_i(\overline{\alpha}_i(s)), s)|s\in I\},$$
which is a $G$-invariant subset in $M\times I$, where the action of $G$ on $I$ is trivial. By Corollary~\ref{one-one},
generally $\mathcal{S}((p^G)^{-1}\overline{\alpha})$ is not an orbit braid in $\mathcal{B}_n^{orb}(M,G)$. Indeed, for any path lifting $\alpha$ of $\overline{\alpha}$, as $G$-invariant subsets of $M\times I$, $\widetilde{c(\alpha)}$ is the same as $\mathcal{S}((p^G)^{-1}\overline{\alpha})$. However, $[\widetilde{c(\alpha)}]$ is uniquely determined by $[\alpha]$, but $\mathcal{S}((p^G)^{-1}\overline{\alpha})$ is not so. In fact, it is possible that there are two path liftings $\alpha$ and $\alpha'$
such that $\alpha(0)=\alpha'(0)$ and $\alpha\not\simeq\alpha'$ rel $\partial I$, so $[\widetilde{c(\alpha)}]\not=[\widetilde{c(\alpha')}]$ but as sets $[\widetilde{c(\alpha)}]=[\widetilde{c(\alpha')}]=\mathcal{S}((p^G)^{-1}\overline{\alpha})$.
\end{remark}

By Corollary~\ref{homotopy version} and  Lemma~\ref{sur},  we then conclude  two epimorphisms
$$\Psi_1: \mathcal{P}_n^{orb}(M,G) \longrightarrow\pi_1^E(F(M/G, n), \overline{\bf x}, \overline{\bf x})=\pi_1(F(M/G,n), \overline{\bf x})=P_n(M/G)$$
and
$$\Psi_2: \mathcal{B}_n^{orb}(M,G) \longrightarrow\pi_1^E(F(M/G, n)/\Sigma_n, \widehat{\overline{\bf x}}, \widehat{\overline{\bf x}})=\pi_1(F(M/G, n)/\Sigma_n, \widehat{\overline{\bf x}})=B_n(M/G)$$
where $\widehat{\overline{\bf x}}$ is the image of $\overline{\bf x}$ under the projection $F(M/G,n)\longrightarrow
F(M/G, n)/\Sigma_n$. Then, making use of Theorem~\ref{main result} again, this gives

\begin{proposition}\label{re-orbits}
There is an epimorphism between two short exact sequences.
\[
\begin{CD}
1@>>> \mathcal{P}_n^{orb}(M,G) @>I_3 >> \mathcal{B}_n^{orb}(M,G) @>>> \Sigma_n @>>> 1\\
@. @VVV   @V  VV  @ V= VV  \\
1@ >>> P_n(M/G) @ > >> B_n(M/G) @>>> \Sigma_n @>>> 1\\
\end{CD}
\]
\end{proposition}

Now let us consider the case in which the action of $G$ on $M$ is free. In this case, the projection $p^G: F_G(M,n)\longrightarrow
F(M/G,n)$ becomes a fibration with fiber $G^{\times n}$.
\begin{lemma}\label{free}
The following statements are equivalent.
\begin{itemize}
 \item[$(1)$]  The action of $G$ on $M$ is free.
 \item[$(2)$]  For any path $\overline{\alpha}: I\longrightarrow F(M/G,n)$, there are exactly $|G|^n$ path liftings of $\overline{\alpha}$.
 \item[$(3)$]  For any path $\overline{\alpha}: I\longrightarrow F(M/G,n)$ and any two path liftings $\alpha'$ and $\alpha''$ of $\overline{\alpha}$, $\widetilde{c(\alpha')}=\widetilde{c(\alpha'')}$.
\end{itemize}
\end{lemma}
\begin{proof}
The equivalence of $(1)$ and $(2)$ is obvious. Assume that there are exactly $|G|^n$ path liftings of $\overline{\alpha}$.
Then we see that for each $1\leq i\leq n$, $(p^G)^{-1}\overline{\alpha}_i$ consists of $|G|$ path liftings of $\overline{\alpha}_i$, all of which do not intersect to each other. Furthermore, we have that for any two path liftings $\alpha'$ and $\alpha''$ of $\overline{\alpha}$, there is some $g\in G^{\times n}$ such that $\alpha'=g\alpha''$, so $\widetilde{c(\alpha')}=\widetilde{c(\alpha'')}$.

\vskip .2cm
Conversely, let $\alpha'$ and $\alpha''$ be two path liftings of $\overline{\alpha}$, and assume that $\widetilde{c(\alpha')}=\widetilde{c(\alpha'')}$. By Lemma~\ref{h-class} and Corollary~\ref{one-one}, without the loss of generality, we may assume that
$\alpha'(0)=\alpha''(0)={\bf x}$. Then $\widetilde{c(\alpha')}=\widetilde{c(\alpha'')}$ implies $\alpha'(1)=\alpha''(1)$. On the other hand, since  $\alpha'$ and $\alpha''$ are two path liftings of $\overline{\alpha}$, there should be some $g\in G^{\times n}$ such that $\alpha'=g\alpha''$.
Since $\alpha'(0)=\alpha''(0)={\bf x}$, this forces $g$ to be the unit element of $G^{\times n}$, so $\alpha'=\alpha''$.
This implies that $\{g\alpha'|g\in G^{\times n}\}$ gives all different path liftings of $\overline{\alpha}$, which consist of exactly
$|G|^n$ path liftings.
\end{proof}

\begin{remark}
Lemma~\ref{free} tells us that if the action of $G$ on $M$ is non-free, then  there must be some path $\overline{\alpha}: I\longrightarrow F(M/G,n)$ and  two different path liftings $\alpha'$ and $\alpha''$ of $\overline{\alpha}$ such that $\widetilde{c(\alpha')}\not=\widetilde{c(\alpha'')}$. In fact, since the action of $G$ on $M$ is non-free, we may assume that
there is some $s\in I$ such that $(p^G)^{-1}(\overline{\alpha}(s))$ is not the free orbit of some point in $F_G(M,n)$. Furthermore,
 there would be more than $|G|^n$ path liftings of $\overline{\alpha}$ since we can do more choices of path liftings of $\overline{\alpha}$ via those points of  non-free orbit in $(p^G)^{-1}(\overline{\alpha}(s))$.
\end{remark}

\begin{corollary}\label{free-2}
If the action of $G$ on $M$ is free, then
\begin{itemize}
 \item[$(1)$]  $\mathcal{P}_n^{orb}(M,G)\cong P_n(M/G)$, so $\mathcal{P}_n^{orb}(M,G)$ is realizable as $\pi_1(F(M/G,n), \overline{\bf x})$.
 \item[$(2)$]  $\mathcal{B}_n^{orb}(M,G)\cong B_n(M/G)$, so $\mathcal{B}_n^{orb}(M,G)$ is realizable as $\pi_1(F(M/G, n)/\Sigma_n, \widehat{\overline{\bf x}})$.
\end{itemize}
\end{corollary}

\begin{proof}
It is a consequence of Lemma~\ref{free} and Proposition~\ref{re-orbits}.
\end{proof}

Compare with Corollary~\ref{free-1}, we obtain that if the action of $G$ on $M$ is free, then
$$Br(M^{\times n}, G^{\times n}\rtimes_\varphi \Sigma_n)\cong B_n(M/G)$$
and
$$Br(M^{\times n}, G^{\times n})\cong P_n(M/G).$$

\begin{remark}
In the viewpoint of the theory of covering spaces, generally $F_G(M,n)$ is not the covering space of $F(M/G,n)$. However,
paths and the homotopies between two paths in $F(M/G,n)$ can still be lifted to $F_G(M,n)$ but liftings with the same staring point may not be unique. Thus, if the action of $G$ on $M$ is non-free, then
the homomorphism $$p^G_*: P_n(M,G)\longrightarrow P_n(M/G)$$ induced by the projection $p^G: F_G(M, n)\longrightarrow F(M/G, n)$
 is no longer injective. Actually, $p^G_*$ is the compostion of a monomorphism and an epimorphism
$$P_n(M,G)\rightarrowtail P_n^\text{orb}(M,G)\twoheadrightarrow P_n(M/G).$$
\end{remark}

\subsection{Relation between orbit configuraiton space and ordinary configuraiton space}
There is a natural embedding $i: F_G(M, n)\hookrightarrow F(M, n)$ from orbit configuration space to its corresponding ordinary configuration space by forgetting the action of $G$.
\begin{lemma}
The induced homomorphism $i_*: \pi_1(F_G(M, n), {\bf x})\longrightarrow \pi_1(F(M, n), {\bf x})$ is an epimorphism.
\end{lemma}
\begin{proof}
Take an element $[\alpha]$ in $\pi_1(F(M, n), {\bf x})$. For any $s\in I$, if $\alpha(s)\in F_G(M,n)$, then $[\alpha]$ is also an element of $\pi_1(F_G(M, n), {\bf x})$.
\vskip .2cm
Now assume that there is some $s\in I$ (possibly $s$ can be any point of the whole $(0,1)$) such that $\alpha(s)\not\in F_G(M,n)$. This means that there are at least two  $i, j$ with $i\not=j$ such that $G(\alpha_i(s))=G(\alpha_j(s))$, where $\alpha=(\alpha_1, ..., \alpha_n)$. Clearly, $\alpha_i(s)$ or $\alpha_j(s)$ is not a $G$-fixed point since $\alpha_i(s)\not=\alpha_j(s)$.
So there exists some $g\not=e$ in $G$ such that $\alpha_i(s)=g\alpha_j(s)$.
Since $G$ is finite, there exists a $G$-invariant open neighborhood $N$ of $\alpha_i(s)$  which is a disjoint union of some connected open subsets in $F(M,n)$
such that for a enough small connected open neighborhood $N_s\subset I$ of $s$,  $\alpha_i(N_s)\cap N$ and $\alpha_j(N_s)\cap N$ lie in two different components of $N$.
Then we can always do a slight homotopy deformation on $\alpha_i$ in $N$, changing $\alpha_i$  into $\alpha'_i$, such that $\alpha'_i(N_s)\cap N$ never meets with the orbits of  other $\alpha_k(N_s)$, $k\not=i$. This gives a change on a small open arc of the path $\alpha_i$ up to homotopy.
\vskip .2cm
  As long as there are also finitely or infinitely many points $s$ in $I$ such that $\alpha_i(s)$ meets some orbit of $\alpha_k(s), k\not=i$, since $\alpha_i(I)$ is compact, we can perform the above approach finite times, so that $\alpha_i$ can be finally changed into a new path $\alpha'_i$ such that $\alpha_i\simeq
\alpha'_i$ rel $\partial I$ in $F(M,n)$, and  for any $s\in I$ and any $k\not=i$, $G(\alpha'_i(s))\cap G(\alpha_k(s))=\emptyset$.
This procedure only does change the component path $\alpha_i$, so $\alpha$ is changed into $(\alpha_1, ..., \alpha_{i-1}, \alpha'_i, \alpha_{i+1}, ..., \alpha_n)$, denoted by $\alpha'$. Clearly, $\alpha\simeq \alpha'$ rel $\partial I$ in $F(M,n)$.

\vskip .2cm
  If $\alpha'$ is not a path in $F_G(M,n)$ yet, then we will perform the above procedure on other component paths $\alpha_k, k\not=i$. Since $c(\alpha)$ only contains $n$ strings,
we can  end our procedure until we obtain a path $\beta$ such that for any $s\in I$, $\beta(s)\in F_G(M,n)$ and
$\alpha\simeq\beta$ rel $\partial I$ in $F(M,n)$. Thus, $[\alpha]$ is in the image of $i_*$. This completes the proof.
\end{proof}

In a similar way as above, we can show that the following homomorphism induced by the embedding $i: F_G(M, n)\hookrightarrow F(M, n)$:
$$\pi_1^E(F_G(M,n), {\bf x}, \Sigma_n({\bf x}))\longrightarrow\pi_1^E(F(M,n), {\bf x}, \Sigma_n({\bf x}))$$
is also an epimorphism. Therefore we have

\begin{proposition}\label{orbit-ordinary}
There is an epimorphism between two short exact sequences.
\[
\begin{CD}
1@>>> \mathcal{P}_n(M,G) @>>> \mathcal{B}_n(M,G) @>>> \Sigma_n @>>> 1\\
@. @VVV   @V  VV  @ V= VV  \\
1@ >>> P_n(M) @ > >> B_n(M) @>>> \Sigma_n @>>> 1\\
\end{CD}
\]
\end{proposition}

\section{Presentations of orbit braid groups in $\mathbb{C}\times I$ with two typical actions on $\mathbb{C}$}\label{calculation}

The geometric presentation of classical braid group $B_n(\mathbb{R}^2)$ in $\mathbb{R}^2\times I$ gives us much more insights to the case of orbit braid group (\cite{A2, Bir}). Thus we begin with  our work from the case of $\mathbb{C} \approx \mathbb{R}^2$ with the following two typical actions:
\begin{itemize}
 \item[(I)]  One is the action $\mathbb{Z}_p\curvearrowright^{\phi_1}\mathbb{C}$
defined by $(e^{{2k\pi{\bf i}}\over p}, z)\longmapsto e^{{2k\pi{\bf i}}\over p}z$, which is non-free and fixes only the origin of $\mathbb{C}$, where $p$ is a prime, and
$\mathbb{Z}_p$ is regarded as the group $\{e^{{2k\pi{\bf i}}\over p}|0\leq k<p\}$. If the action $\phi_1$ is restricted to $\mathbb{C}^{\times}=\mathbb{C}\setminus\{0\}$, then the action $\mathbb{Z}_p\curvearrowright^{\phi_1}\mathbb{C}^{\times}$ is free.
\item[(II)]
 The other one is a non-free action $(\mathbb{Z}_2)^2\curvearrowright^{\phi_2}\mathbb{C}$ defined by
$$\begin{cases}
z\longmapsto \overline{z}\\
z\longmapsto -\overline{z}.
\end{cases}
$$
This action is just the standard representation of $(\mathbb{Z}_2)^2$ on $\mathbb{C}\approx \mathbb{R}^2$.
\end{itemize}

 Throughout the following, fix
$${\bf z}=(1+{\bf i}, 2+2{\bf i}, ..., n+n{\bf i})$$ as the base point in $F_{\mathbb{Z}_p}(\mathbb{C},n)$, $F_{\mathbb{Z}_p}(\mathbb{C}^{\times},n)$ and $F_{(\mathbb{Z}_2)^2}(\mathbb{C},n)$, where ${\bf i}=\sqrt{-1}$. Clearly, whichever the action is, the orbit of ${\bf z}$ is of free. For a convenience, by $\widehat{l}$ we mean $l+l{\bf i}$, so
${\bf z}=(1+{\bf i}, 2+2{\bf i}, ..., n+n{\bf i})=(\widehat{1}, \widehat{2}, ..., \widehat{n})$.

\subsection{Orbit braid group of $F_{\mathbb{Z}_p}(\mathbb{C},n)$}

For a path $\alpha=(\alpha_1, ...,\alpha_n)$ in $F_{\mathbb{Z}_2}( \mathbb{C}, n )$ with $\alpha(0)={\bf z}$ and $\alpha(1)=g{\bf z}_{\sigma}$ where $g\in \mathbb{Z}_p$, it is easy to see that
the corresponding orbit braid $\widetilde{c(\alpha)}=\coprod_{i=1}^n \widetilde{c(\alpha_i)}$ is symmetric with respect to the line $O=\{0\}\times I$ in $\mathbb{C}\times I$, where $$\widetilde{c(\alpha_i)}=\{(\alpha_i(s), s), (e^{{2\pi{\bf i}}\over p}\alpha_i(s), s), \cdots, (e^{{2(p-1)\pi{\bf i}}\over p}\alpha_i(s), s)|s\in I\}.$$

\vskip .2cm
First let us consider the case $p=2$. To describe $\mathcal{B}_n^{\text{orb}}(\mathbb{C} ,\mathbb{Z}_2)$, we construct a family of basic "bricks"  ${\bf b}_k, k=1, ...,n-1,$ and
${\bf b}$, where each orbit braid ${\bf b}_k$  is chosen as the class $[\widetilde{c(\alpha^{(k)})}]$ given by the path
 \begin{equation}\label{path-1}
 \alpha^{(k)}(s)=(1+{\bf i},\dots,k+(k+1){\bf i}+\exp^{-\frac{\pi}{2}{\bf i}(1-s)},(k+1)+k{\bf i}+{\bf i}\exp^{\frac{\pi}{2}{\bf i}s},\dots,n+n{\bf i})
 \end{equation}
as shown in the following picture

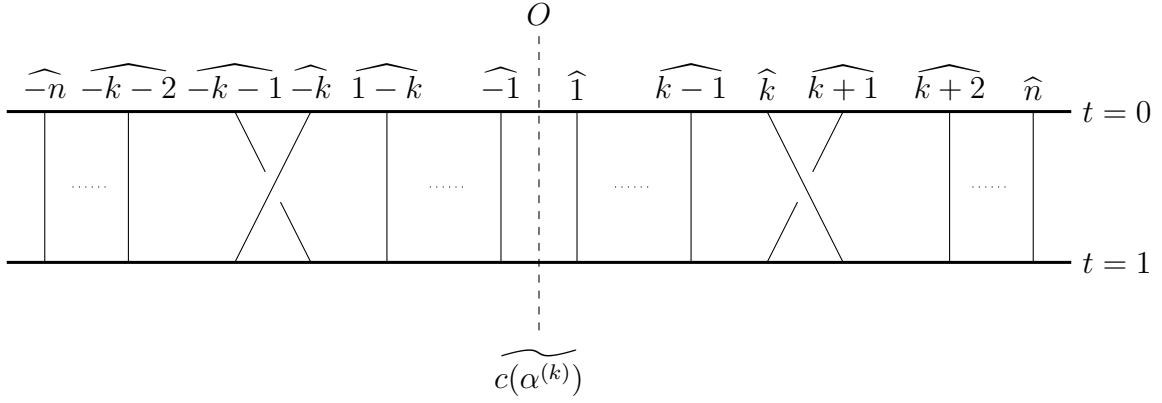
\begin{tikzpicture}
\draw [line width=0.04cm] (-7,0)--(7,0);
\node[right] at (7,0) {$t=0$};
\draw [line width=0.04cm] (-7,-2)--(7,-2);
\node[right] at (7,-2) {$t=1$};
\draw [dashed] (0,1)--(0,-3);
\node
 [above] at (0,1) {$O$};
\draw (0.5,0)--(0.5,-2);
\node[above] at (0.5,0) {$\widehat{1}$};
\draw (-0.5,0)--(-0.5,-2);
\node[above] at (-0.5,0) {$\widehat{-1}$};
\draw[dotted] (1,-1)--(1.5,-1);
\draw[dotted] (-1,-1)--(-1.5,-1);
\draw (2,0)--(2,-2);
\node[above] at (2,0) {$\widehat{k-1}$};
\draw (-2,0)--(-2,-2);
\node[above] at (-2,0) {$\widehat{1-k}$};

\draw (3,0)--(4,-2);
\node[above] at (3,0) {$\widehat{k}$};
\draw (-3,0)--(-4,-2);
\node[above] at (-3,0) {$\widehat{-k}$};

\draw (4,0)--(3.6,-0.8);
\draw (3.4,-1.2)--(3,-2);
\node[above] at (4,0) {$\widehat{k+1}$};
\draw (-4,0)--(-3.6,-0.8);
\draw (-3.4,-1.2)--(-3,-2);
\node[above] at (-4,0) {$\widehat{-k-1}$};

\draw (5.4,0)--(5.4,-2);
\node[above] at (5.4,0) {$\widehat{k+2}$};
\draw (-5.4,0)--(-5.4,-2);
\node[above] at (-5.4,0) {$\widehat{-k-2}$};

\draw (6.5,0)--(6.5,-2);
\node[above] at (6.5,0) {$\widehat{n}$};
\draw (-6.5,0)--(-6.5,-2);
\node[above] at (-6.5,0) {$\widehat{-n}$};

\draw[dotted] (5.7,-1)--(6.2,-1);
\draw[dotted] (-5.7,-1)--(-6.2,-1);
\node[below] at (0,-3) {$\widetilde{c(\alpha^{(k)})}$};
\end{tikzpicture}

\noindent and ${\bf b}$ is chosen as the class $[\widetilde{c(\alpha)}]$ given by the path
\begin{equation}\label{path-2}
\alpha(s)=((1+{\bf i})(1-2s), 2+2{\bf i},\dots, n+n{\bf i})
\end{equation}
 as shown in the following picture

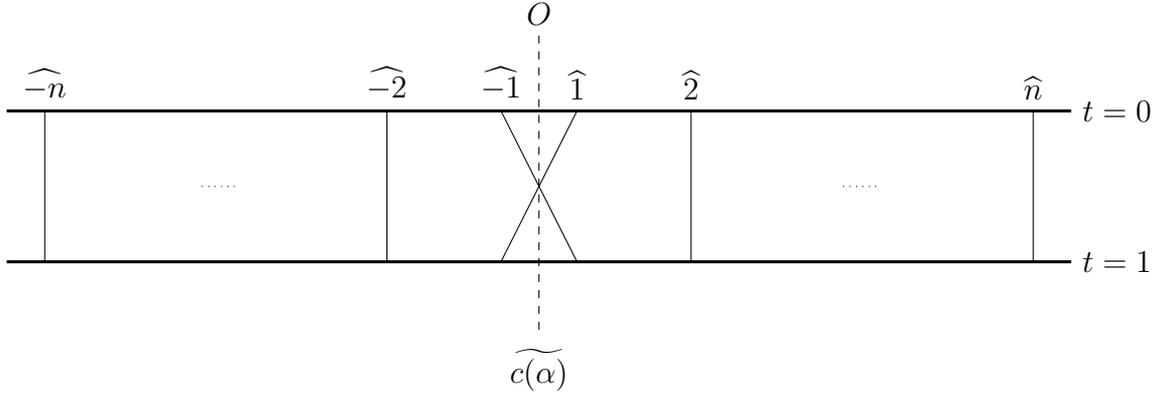
\begin{tikzpicture}
\draw [line width=0.04cm] (-7,0)--(7,0);
\node[right] at (7,0) {$t=0$};
\draw [line width=0.04cm] (-7,-2)--(7,-2);
\node[right] at (7,-2) {$t=1$};
\draw [dashed] (0,1)--(0,-3);
\node[above] at (0,1) {$O$};
\draw (0.5,0)--(-0.5,-2);
\node[above] at (0.5,0) {$\widehat{1}$};
\draw (-0.5,0)--(0.5,-2);
\node[above] at (-0.5,0) {$\widehat{-1}$};

\draw (2,0)--(2,-2);
\node[above] at (2,0) {$\widehat{2}$};
\draw (-2,0)--(-2,-2);
\node[above] at (-2,0) {$\widehat{-2}$};

\draw (6.5,0)--(6.5,-2);
\node[above] at (6.5,0) {$\widehat{n}$};
\draw (-6.5,0)--(-6.5,-2);
\node[above] at (-6.5,0) {$\widehat{-n}$};

\draw[dotted] (4,-1)--(4.5,-1);
\draw[dotted] (-4,-1)--(-4.5,-1);

\node[below] at (0,-3) {$\widetilde{c(\alpha)}$};
\end{tikzpicture}

\begin{remark}\label{intersect}
	In the above picture for $\widetilde{c(\alpha)}$, we see that the first string  and its orbit can exactly intersect at $O$. The reason that this can happen is because the  origin in $\mathbb{C}$ is just a fixed point of $\mathbb{Z}_2$-action. However, this is not necessary. Actually,  even if the first string and its orbit do not intersect at $O$, then the corresponding orbit braid can still be equivalent to $\widetilde{c(\alpha)}$. In fact, we can choose the following path
\begin{equation}\label{path-3}
\beta(s)=((1+{\bf i})e^{\pi{\bf i}s}, 2+2{\bf i},\dots, n+n{\bf i}),
\end{equation}	
which never goes through the origin of $\mathbb{C}$, but $[\widetilde{c(\alpha)}]=[\widetilde{c(\beta)}]$. Thus, ${\bf b}$ can also be chosen as $[\widetilde{c(\beta)}]$.
	\vskip .2cm
	In addition, we note that there are also other orbit braids $\widetilde{c(\gamma)}$ in which the $i$-th string $c(\gamma_i)$ and its orbit intersect or not but their ending points are exchanged, $i=2, ..., n$. However, these orbit braids are not basic "bricks". In fact, we see easily that  for each $i$, $[\widetilde{c(\gamma)}]$ can be represented as a composition  ${\bf b}_{i-1}^{-1}\cdots {\bf b}_1^{-1}{\bf b}{\bf b}_1\cdots {\bf b}_{i-1}$. The following two pictures illustrate  the case of $i=2$.
	
	
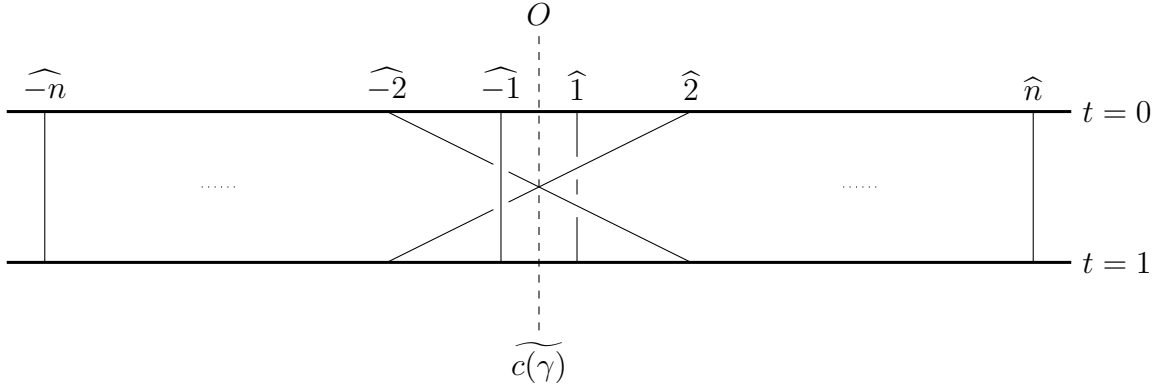
\begin{tikzpicture}
	\draw [line width=0.04cm] (-7,0)--(7,0);
	\node[right] at (7,0) {$t=0$};
	\draw [line width=0.04cm] (-7,-2)--(7,-2);
	\node[right] at (7,-2) {$t=1$};
	\draw [dashed] (0,1)--(0,-3);
	\node[above] at (0,1) {$O$};
	\draw (0.5,0)--(0.5,---0.6);
	\draw (0.5,-0.9)--(0.5,-1.1);
	\draw (0.5,-1.4)--(0.5,-2);
	\node[above] at (0.5,0) {$\widehat{1}$};
	
	\draw (-0.5,0)--(-0.5,-2);
	\node[above] at (-0.5,0) {$\widehat{-1}$};
	
	\draw (2,0)--(-0.4,-1.2);
	\draw (-0.6,-1.3)--(-2,-2);
	\node[above] at (2,0) {$\widehat{2}$};
	\draw (-2,0)--(-0.6,-0.7);
	\draw (-0.4,-0.8)--(2,-2);
	\node[above] at (-2,0) {$\widehat{-2}$};
	
	\draw (6.5,0)--(6.5,-2);
	\node[above] at (6.5,0) {$\widehat{n}$};
	\draw (-6.5,0)--(-6.5,-2);
	\node[above] at (-6.5,0) {$\widehat{-n}$};
	
	\draw[dotted] (4,-1)--(4.5,-1);
	\draw[dotted] (-4,-1)--(-4.5,-1);
\node[below] at (0,-2.9) {$\widetilde{c(\gamma)}$};
	\end{tikzpicture}

\noindent and
	
	
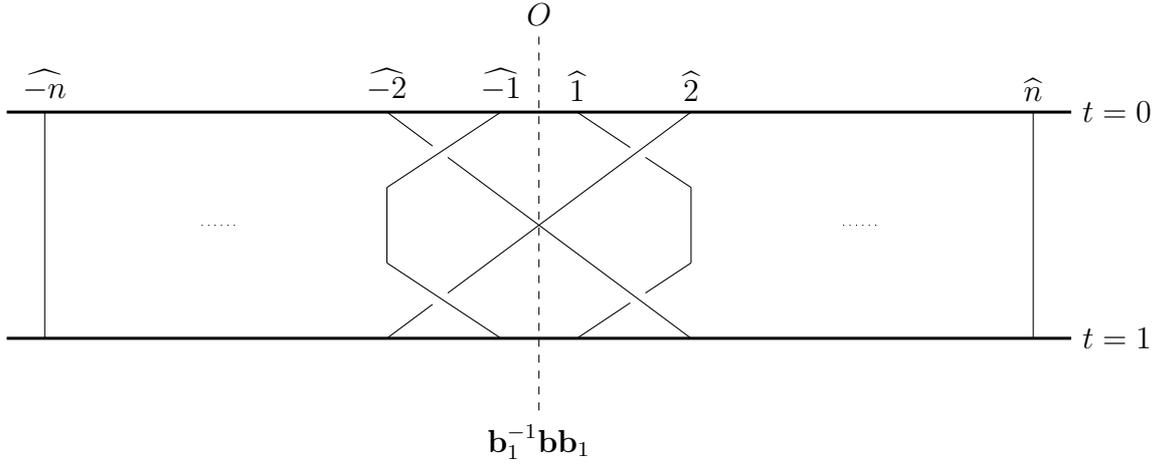
\begin{tikzpicture}
\draw [line width=0.04cm] (-7,0)--(7,0);
\node[right] at (7,0) {$t=0$};
\draw [line width=0.04cm] (-7,-3)--(7,-3);
\node[right] at (7,-3) {$t=1$};
\draw [dashed] (0,1)--(0,-4);
\node[above] at (0,1) {$O$};

\draw (0.5,0)--(1.2,---0.47);
\draw (1.4,-0.6)--(2,-1);
\draw (2,-1)--(2,-2);
\draw (2,-2)--(1.4,-2.4);
\draw (1.2,-2.53)--(0.5,-3);
\node[above] at (0.5,0) {$\widehat{1}$};
\draw (-0.5,0)--(-2,-1);
\draw (-2,-1)--(-2,-2);
\draw (-2,-2)--(-0.5,-3);
\node[above] at (-0.5,0) {$\widehat{-1}$};

\draw (2,0)--(-1.2,-2.4);
\draw (-1.4,-2.55)--(-2,-3);
\node[above] at (2,0) {$\widehat{2}$};
\draw (-2,0)--(-1.4,-0.45);
\draw (-1.2,-0.6)--(2,-3);
\node[above] at (-2,0) {$\widehat{-2}$};

\draw (6.5,0)--(6.5,-3);
\node[above] at (6.5,0) {$\widehat{n}$};
\draw (-6.5,0)--(-6.5,-3);
\node[above] at (-6.5,0) {$\widehat{-n}$};

\draw[dotted] (4,-1.5)--(4.5,-1.5);
\draw[dotted] (-4,-1.5)--(-4.5,-1.5);
\node[below] at (0,-4) {${\bf b}_1^{-1}{\bf b}{\bf b}_1$};
\end{tikzpicture}
\end{remark}

\begin{proposition} \label{example-1}
$\mathcal{B}_n^{\text{orb}}(\mathbb{C} ,\mathbb{Z}_2)$ is generated by ${\bf b}_k$ $(1\leq k\leq n-1)$ and ${\bf b}$, with relations
\begin{itemize}
\item[$(1)$] ${\bf b}^2=e$;
\item[$(2)$]${\bf b}{\bf b}_1{\bf b}{\bf b}_1={\bf b}_1{\bf b}{\bf b}_1{\bf b}$;
\item[$(3)$]${\bf b}_k{\bf b}={\bf b}{\bf b}_k \quad( k>1 )$;
\item[$(4)$]${\bf b}_k{\bf b}_{k+1}{\bf b}_k={\bf b}_{k+1}{\bf b}_k{\bf b}_{k+1}$;
\item[$(5)$]${\bf b}_k{\bf b}_l={\bf b}_l{\bf b}_k \quad(|k-l|>1)$.
\end{itemize}
\end{proposition}
\begin{proof}
First, it is obvious that  every class in $\mathcal{B}_n^{\text{orb}}(\mathbb{C} ,\mathbb{Z}_2)$ can be reduced into the composition
of ${\bf b}_k$ and ${\bf b}$ because each crossing of two adjacent orbit strings just decides a basic "brick".

\vskip .2cm
Each ${\bf b}_k$ has a symmetric structure with respect to $O$ and its half part is used as a generator of classical braid group $B_n(\mathbb{R}^2)$. Thus, the relations $(4)$ and $(5)$ follow from the theory of classical braid groups (see, e.g., \cite{Bir}).

\vskip .2cm

We can construct homotopy deformation maps that connect both sides of the equations in relations $(1),(2)$, and $(3)$, respectively.
Actually, we can intuitively see this. Let us look at the pictures of orbit braids in both sides of relation $(2)$, as shown below.

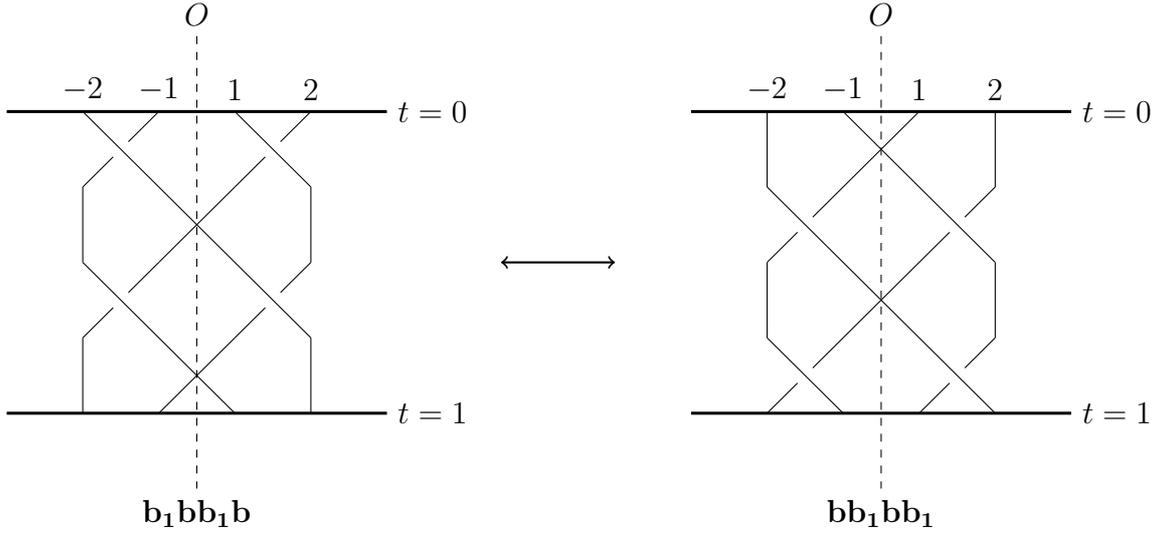
\begin{tikzpicture}
\draw [line width=0.04cm] (-7,0)--(-2,0);
\draw [line width=0.04cm] (2,0)--(7,0);
\node[right] at (-2,0) {$t=0$};
\node[right] at (7,0) {$t=0$};
\draw [line width=0.04cm] (-7,-4)--(-2,-4);
\draw [line width=0.04cm] (2,-4)--(7,-4);
\node[right] at (7,-4) {$t=1$};
\node[right] at (-2,-4) {$t=1$};
\draw [dashed] (-4.5,1)--(-4.5,-5);
\node[above] at (-4.5,1) {$O$};
\draw [dashed] (4.5,1)--(4.5,-5);
\node[above] at (4.5,1) {$O$};

\draw (-4,0)--(-3,-1);
\draw (-3,-1)--(-3,-2);
\draw (-3,-2)--(-3.4,-2.4);
\draw (-3.6,-2.6)--(-4,-3);
\draw (-4,-3)--(-5,-4);
\node[above] at (-4,0) {$1$};

\draw (-3,0)--(-3.4,-0.4);
\draw (-3.6,-0.6)--(-4,-1);
\draw (-4,-1)--(-5,-2);
\draw (-5,-2)--(-5.4,-2.4);
\draw (-5.6,-2.6)--(-6,-3);
\draw (-6,-3)--(-6,-4);
\node[above] at (-3,0) {$2$};

\draw (-5,0)--(-5.4,-0.4);
\draw (-5.6,-0.6)--(-6,-1);
\draw (-6,-1)--(-6,-2);
\draw (-6,-2)--(-4,-4);
\node[above] at (-5,0) {$-1$};

\draw (-6,0)--(-3,-3);
\draw (-3,-3)--(-3,-4);
\node[above] at (-6,0) {$-2$};
\node [below] at (-4.5,-5) {${\bf b_1bb_1b}$};

\draw [<->, thick] (-0.5,-2)--(1,-2);

\draw (5,0)--(3.6,-1.4);
\draw (3.4,-1.6)--(3,-2);
\draw (3,-2)--(3,-3);
\draw (3,-3)--(4,-4);
\node[above] at (5,0) {$1$};

\draw (6,0)--(6,-1);
\draw (6,-1)--(5.6,-1.4);
\draw (5.4,-1.6)--(3.6,-3.4);
\draw (3.4,-3.6)--(3,-4);
\node[above] at (6,0) {$2$};

\draw (4,0)--(6,-2);
\draw (6,-2)--(6,-3);
\draw (6,-3)--(5.6,-3.4);
\draw (5.4,-3.6)--(5,-4);
\node[above] at (4,0) {$-1$};

\draw (3,0)--(3,-1);
\draw (3,-1)--(6,-4);
\node[above] at (3,0) {$-2$};
\node [below] at (4.5,-5) {${\bf bb_1bb_1}$};
\end{tikzpicture}

\noindent Since we can always do a slight homotopy deformation on ${\bf b}$ near $O$ such that the first string  and its orbit  do not intersect at $O$ as stated in Remark~\ref{intersect}, two crossings at $O$ of the left orbit braid (or right orbit braid) in the above picture can be exchanged. This illustrates the equivalence of two orbit braids. Actually, if we choose $\widetilde{c(\beta)}$ in Remark~\ref{intersect} as a representative of ${\bf b}$, then we can just avoid this obstruction produced by using $\widetilde{c(\alpha)}$ as a representative of ${\bf b}$.

\vskip .2cm
We can also use a similar way to do this for the cases of $(1)$ and $(3)$. However,  we would like to leave them as  exercises to the reader.
\end{proof}

For the general prime $p$,  we first need to modify the path $\alpha$ in (\ref{path-2}) or $\beta$ in (\ref{path-3})
into the general form
\begin{equation}\label{path-4}
\alpha(s)=((1+{\bf i})(1+(e^{{2\pi{\bf i}}\over p}-1)s), 2+2{\bf i},\dots, n+n{\bf i}).
\end{equation}
or
\begin{equation}\label{path-5}
\beta(s)=((1+{\bf i})e^{{2\pi{\bf i}s}\over p}, 2+2{\bf i},\dots, n+n{\bf i}).
\end{equation}
Then we can  use the paths  $\alpha^{(k)}$ in (\ref{path-1}) and  $\alpha$ in (\ref{path-4}) (or $\beta$ in (\ref{path-5})) to construct the required basic "bricks" ${\bf b}_k=[\widetilde{c(\alpha^{(k)})}]$ for $1\leq k\leq n-1$
and ${\bf b}= [\widetilde{c(\alpha)}]$ or $[\widetilde{c(\beta)}]$, each of which would consist of  $p$  symmetric parts with respect to the line $O$.

\vskip .2cm
It is not difficult to see that each class in $\mathcal{B}_n^{\text{orb}}(\mathbb{C} ,\mathbb{Z}_p)$ is also a composition of ${\bf b}_k (1\leq k\leq n-1)$ and ${\bf b}$. To get the presentation of $\mathcal{B}_n^{\text{orb}}(\mathbb{C} ,\mathbb{Z}_p)$, an easy observation shows that we merely need to change the relations (1) and (2) in Proposition~\ref{example-1} into ${\bf b}^p=e$ and $({\bf b}{\bf b}_1)^p=({\bf b}_1{\bf b})^p$. The key reason for ${\bf b}^p=e$ is that any twine of an orbit string surround the line $O$ can be undone since the origin of $\mathbb{C}$ is a fixed point of the action. For $({\bf b}{\bf b}_1)^p=({\bf b}_1{\bf b})^p$, if we choose $\widetilde{c(\beta)}$ as a representative of ${\bf b}$, then  $({\bf b}{\bf b}_1)^p=({\bf b}_1{\bf b})^p$ can be seen intuitively without any obstruction. Of course, the reason for ${\bf b}^p=e$ is different from that for $({\bf b}{\bf b}_1)^p=({\bf b}_1{\bf b})^p$.
Thus we have that

\begin{proposition}
	$\mathcal{B}_n^{\text{orb}}(\mathbb{C} ,\mathbb{Z}_p)$ is generated by ${\bf b}_k$ $(1\leq k\leq n-1)$ and ${\bf b}$, with relations
	\begin{itemize}
		\item[$(1)$] ${\bf b}^p=e$;
		\item[$(2)$]$({\bf b}{\bf b}_1)^p=({\bf b}_1{\bf b})^p$;
		\item[$(3)$]${\bf b}_k{\bf b}={\bf b}{\bf b}_k \quad( k>1 )$;
		\item[$(4)$]${\bf b}_k{\bf b}_{k+1}{\bf b}_k={\bf b}_{k+1}{\bf b}_k{\bf b}_{k+1}$;
		\item[$(5)$]${\bf b}_k{\bf b}_l={\bf b}_l{\bf b}_k \quad(|k-l|>1)$.
	\end{itemize}
\end{proposition}

\subsection{Orbit braid group of $F_{\mathbb{Z}_p}(\mathbb{C}^\times,n)$}
In the similar way to the case of $B_n^{\text{orb}}(\mathbb{C},\mathbb{Z}_p)$, we can describe $B_n^{\text{orb}}(\mathbb{C}^\times,\mathbb{Z}_p)$. In this case,  a family of basic "bricks" named after ${\bf b}_k$ and ${\bf b'}$ can also be constructed by the paths  $\alpha^{(k)}$ in (\ref{path-1}) and $\beta$ in (\ref{path-5}). Here we make sure by using $\widetilde{c(\beta)}$ as a representative of ${\bf b}'$ that $p$ ordinary strings of the first orbit string in $\widetilde{c(\beta)}$ must  not intersect to each other because  $\mathbb{Z}_p$ acts freely on $\mathbb{C}^\times$, as shown in the following picture for the case of $p=2$:

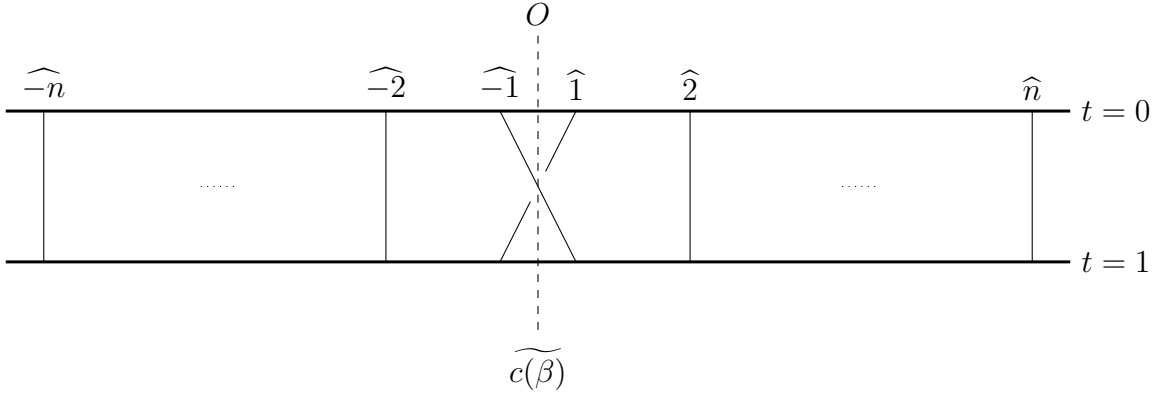
\begin{tikzpicture}
\draw [line width=0.04cm] (-7,0)--(7,0);
\node[right] at (7,0) {$t=0$};
\draw [line width=0.04cm] (-7,-2)--(7,-2);
\node[right] at (7,-2) {$t=1$};
\draw [dashed] (0,1)--(0,-3);
\node[above] at (0,1) {$O$};
\draw (0.5,0)--(0.1,-0.8);
\draw (-0.1,-1.2)--(-0.5,-2);
\node[above] at (0.5,0) {$\widehat{1}$};
\draw (-0.5,0)--(0.5,-2);
\node[above] at (-0.5,0) {$\widehat{-1}$};

\draw (2,0)--(2,-2);
\node[above] at (2,0) {$\widehat{2}$};
\draw (-2,0)--(-2,-2);
\node[above] at (-2,0) {$\widehat{-2}$};

\draw (6.5,0)--(6.5,-2);
\node[above] at (6.5,0) {$\widehat{n}$};
\draw (-6.5,0)--(-6.5,-2);
\node[above] at (-6.5,0) {$\widehat{-n}$};

\draw[dotted] (4,-1)--(4.5,-1);
\draw[dotted] (-4,-1)--(-4.5,-1);

\node[below] at (0,-3) {$\widetilde{c(\beta)}$};
\end{tikzpicture}

Clearly, ${\bf b}_k$ and ${\bf b'}$ have the same expressions in  $B_n^{\text{orb}}(\mathbb{C},\mathbb{Z}_p)$ and $B_n^{\text{orb}}(\mathbb{C}^\times,\mathbb{Z}_p)$.

\vskip .2cm
On ${\bf b}_k$, we see that there is not any direct twine among $p$ symmetric parts with respect to $O$ in $\widetilde{c(\alpha^{(k)})}$, and only thing that happens is that two strings within each symmetric part do an exchange of ending points. So  ${\bf b}_k$ in  $B_n^{\text{orb}}(\mathbb{C},\mathbb{Z}_p)$ has the same property as  in $B_n^{\text{orb}}(\mathbb{C}^\times,\mathbb{Z}_p)$.

\vskip .2cm
Since we are working on the case of $\mathbb{C}^\times$ with a free $\mathbb{Z}_p$-action, this means that we cannot undo any twine of an orbit string surround the line $O$  of an orbit braid, so
${\bf b'}$ in $B_n^{\text{orb}}(\mathbb{C}^\times,\mathbb{Z}_p)$ would have an essential difference from ${\bf b}'$ in $B_n^{\text{orb}}(\mathbb{C},\mathbb{Z}_2)$.
However, this essential difference of ${\bf b}'$ in $B_n^{\text{orb}}(\mathbb{C}^\times,\mathbb{Z}_p)$
and $B_n^{\text{orb}}(\mathbb{C},\mathbb{Z}_p)$  only brings a little bit  difference  on the structures of two orbit braid groups.
Actually, the only difference is that the order of ${\bf b}'$ in $B_n^{\text{orb}}(\mathbb{C}^\times,\mathbb{Z}_p)$ is different from that of ${\bf b}'$ in  $B_n^{\text{orb}}(\mathbb{C},\mathbb{Z}_p)$. We see by a direct observation that ${\bf b}'$ in $B_n^{\text{orb}}(\mathbb{C}^\times,\mathbb{Z}_p)$ is an element of infinite order. Thus we have that

\begin{proposition}
	$B_n^{\text{orb}}(\mathbb{C}^\times ,\mathbb{Z}_p)$ is genereated by ${\bf b_k}$ $(1\leq k\leq n-1)$ and ${\bf b'}$, with relations:
	\begin{itemize}
		\item[$(1)$]$({\bf b'}{\bf b}_1)^p=({\bf b}_1{\bf b'})^p$;
		\item[$(2)$]${\bf b}_k{\bf b'}={\bf b'}{\bf b}_k \quad( k>1 )$;
		\item[$(3)$]${\bf b}_k{\bf b}_{k+1}{\bf b}_k={\bf b}_{k+1}{\bf b}_k{\bf b}_{k+1}$;
		\item[$(4)$]${\bf b}_k{\bf b}_l={\bf b}_l{\bf b}_k \quad(|k-l|>1)$.
	\end{itemize}
\end{proposition}

\subsection{Orbit braid group of $F_{\mathbb{Z}^2_2}(\mathbb{C},n)$}

Identify $\mathbb{C}$ with $\mathbb{R}^2$, the action $(\mathbb{Z}_2)^2\curvearrowright^{\phi_2}\mathbb{C}$ is just the standard
$\mathbb{Z}_2^2$-representation on $\mathbb{R}^2$, such that $\mathbb{Z}_2^2$ is generated by two reflections $g^x$  and $g^y$  with respect to $x$-axis and $y$-axis, respectively. For a path $\alpha=(\alpha_1,\dots,\alpha_n)$ in $F_{\mathbb{Z}_2^2}(\mathbb{C},n)$ with $\alpha(0)={\bf z}$ and $\alpha(1)=g{\bf z}_\sigma$, the correspoding orbit braid $\widetilde{c(\alpha)}=\coprod_{i=1}^n \widetilde{c(\alpha_i)}$ is symmetric with respect to the line $O$, $x$-axis$\times I$  and $y$-axis$\times I$ in $\mathbb{C}\times I$, where $$\widetilde{c(\alpha_i)}=\{(\alpha_i(s), s), (-\alpha_i(s), s),(\overline{\alpha_i(s)}, s),(-\overline{\alpha_i(s)}, s),|s\in I\}$$ and $\overline{\alpha_i(s)}$ means the conjugacy   of $\alpha_i(s)$.

\vskip .2cm
Based upon the symmetries of the orbit braids in $B_n^{\text{orb}}(\mathbb{C},\mathbb{Z}_2^2)$,  we construct a family of basic "bricks" named after ${\bf b}_k$,  ${\bf b}^x$ and ${\bf b}^y$ as follows:
\begin{itemize}
\item[$(1)$]  ${\bf b}_k$ is chosen as $[\widetilde{c(\alpha^{(k)})}]$ where $\alpha^{(k)}$ is the path in (\ref{path-1});
\item[$(2)$]  ${\bf b}^x$ is chosen as $[\widetilde{c(\alpha^x)}]$ where $\alpha^x$ is the path given by
$$\alpha^x(s)=(1+(1-2s){\bf i}, 2+2{\bf i}, ..., n+n{\bf i})$$
such that $\alpha^x_1$ and $\overline{\alpha^x_1}$ intersect at $x$-axis$\times I$;
\item[$(3)$] ${\bf b}^y$ is chosen as $[\widetilde{c(\alpha^y)}]$ where $\alpha^y$ is the path given by
$$\alpha^y(s)=((1-2s)+{\bf i}, 2+2{\bf i}, ..., n+n{\bf i})$$
such that $\alpha^y_1$ and $-\overline{\alpha^y_1}$ intersect at $y$-axis$\times I$.
\end{itemize}

\begin{proposition}
$B_n^{\text{orb}}(\mathbb{C},\mathbb{Z}_2^2)$ is genereated by ${\bf b}_k$ $(1\leq k\leq n-1)$,  ${\bf b}^x$ and ${\bf b}^y$ with relations
\begin{itemize}
\item[$(1)$] $({\bf b}^x)^2=({\bf b}^y)^2=e$;
\item[$(2)$]${\bf b}^x{\bf b}^y={\bf b}^y{\bf b}^x$;
\item[$(3)$]${\bf b}^x{\bf b}_1{\bf b}^x{\bf b}_1={\bf b}_1{\bf b}^x{\bf b}_1{\bf b}^x, \quad {\bf b}^y{\bf b}_1{\bf b}^y{\bf b}_1={\bf b}_1{\bf b}^y{\bf b}_1{\bf b}^y $;
\item[$(4)$]${\bf b}_k{\bf b}^x={\bf b}^x{\bf b}_k, \quad  {\bf b}_k{\bf b}^y={\bf b}^y{\bf b}_k \quad(k>1)$;
\item[$(5)$]${\bf b}_k{\bf b}_{k+1}{\bf b}_k={\bf b}_{k+1}{\bf b}_k{\bf b}_{k+1}$;
\item[$(6)$]${\bf b}_k{\bf b}_l={\bf b}_l{\bf b}_k \quad (|k-l|>1)$.
\end{itemize}
\end{proposition}

\begin{proof}
The proof is similar to that of Proposition~\ref{example-1} with only more cases involved. We would like to leave it as an exercise to reader.
\end{proof}

\subsection{Relation with generalized braid group}

Require every generator in $\mathcal{B}_n^{\text{orb}}(\mathbb{C} ,\mathbb{Z}_2)$ or $B_n^{\text{orb}}(\mathbb{C}^\times,\mathbb{Z}_2)$ be of order 2, we get the transformation group of
ending points of orbit braids, which is $\mathbb{Z}_2^n\rtimes_\varphi \Sigma_n $.
 An easy argument shows that $\mathbb{Z}_2^n\rtimes_\varphi \Sigma_n$ is exactly isomorphic to finite Coxeter group $B_n$.
So, by Theorem~\ref{main result} we have the following short exact sequences:
\begin{equation}\label{e-1}
1\longrightarrow P_n(\mathbb{C},\mathbb{Z}_2)\longrightarrow B_n^{\text{orb}}(\mathbb{C},\mathbb{Z}_2)\longrightarrow B_n \longrightarrow 1
\end{equation}
and
\begin{equation}\label{e-2}
1\longrightarrow P_n(\mathbb{C}^\times,\mathbb{Z}_2)\longrightarrow B_n^{\text{orb}}(\mathbb{C}^\times,\mathbb{Z}_2)\longrightarrow B_n \longrightarrow 1.
\end{equation}
Since we have obtained the presentations of $B_n^{\text{orb}}(\mathbb{C},\mathbb{Z}_2)$ and $B_n^{\text{orb}}(\mathbb{C}^\times,\mathbb{Z}_2)$, we can naturally consider the calculations of $P_n(\mathbb{C},\mathbb{Z}_2)$ and $P_n(\mathbb{C}^\times,\mathbb{Z}_2)$.  Here we only pay our attention on the relations of these orbit braid groups with generalized braid groups, for the concept of generalized braid group,  see Appendix A.

\vskip .2cm
It was known in \cite{G} that two orbit configuration spaces $F_{\mathbb{Z}_2}(\mathbb{C},n)$ and $F_{\mathbb{Z}_2}(\mathbb{C}^\times,n)$ are classifying space of two generalized pure braid groups $P(D_n)$ and $P(B_n)$.
This means that the actions of $D_n$ and $B_n$ on $F_{\mathbb{Z}_2}(\mathbb{C},n)$ and $F_{\mathbb{Z}_2}(\mathbb{C}^\times,n)$ respectively are free. Of course, it is obvious that the action of  $B_n$ on $F_{\mathbb{Z}_2}(\mathbb{C}^\times,n)$ is free
since the action of $\mathbb{Z}_2$ on $\mathbb{C}^\times$ is free. With this point of view,  we have the following two short exact sequences:
\begin{equation}\label{e-3}
1\longrightarrow P(D_n)\longrightarrow Br(D_n)\longrightarrow D_n \longrightarrow 1
\end{equation}
and
\begin{equation}\label{e-4}
1\longrightarrow P(B_n)\longrightarrow Br(B_n)\longrightarrow B_n \longrightarrow 1.
\end{equation}
In addition, we also have that
$$P_n(\mathbb{C},\mathbb{Z}_2)\cong P(D_n)\cong \pi_1(F_{\mathbb{Z}_2}(\mathbb{C},n), {\bf z})$$ and $$P_n(\mathbb{C}^\times,\mathbb{Z}_2)\cong P(B_n)\cong \pi_1(F_{\mathbb{Z}_2}(\mathbb{C}^\times,n), {\bf z}).$$
Moreover, it should be interesting to know how there are the relations between the orbit braid groups $\mathcal{B}_n^{\text{orb}}(\mathbb{C},\mathbb{Z}_2)$, $B_n^{\text{orb}}(\mathbb{C}^\times,\mathbb{Z}_2)$ and the generalised braid groups $Br(D_n)$, $Br(B_n)$.
\vskip .2cm

First let us look at the case of $F_{\mathbb{Z}_2}(\mathbb{C}^\times,n)$. Since $P_n(\mathbb{C}^\times,\mathbb{Z}_2)\cong P(B_n)$, it follows that two short exact sequences (\ref{e-2}) and (\ref{e-4}) are essentially the same. Thus
 we have that
 \begin{proposition}
  $B_n^{\text{orb}}(\mathbb{C}^\times,\mathbb{Z}_2)$ is  isomorphic to the generalized braid group $Br(B_n)$.
  \end{proposition}
  This can also be seen from  Corollary~\ref{free-1}. In this case, the orbit braid group agrees with the generalized braid group.

  \vskip .2cm

  As for the case of $F_{\mathbb{Z}_2}(\mathbb{C}, n)$, compare two short exact sequences (\ref{e-1}) and (\ref{e-3}),
  we see that $P_n(\mathbb{C},\mathbb{Z}_2)\cong P(D_n)$ but $B_n\not\cong D_n$, so $\mathcal{B}_n^{\text{orb}}(\mathbb{C},\mathbb{Z}_2)$ and $Br(D_n)$ are not isomorphic.
  Next let us analyze the connection between $\mathcal{B}_n^{\text{orb}}(\mathbb{C},\mathbb{Z}_2)$ and $Br(D_n)$.

\vskip .2cm
It is well-known that the finite Coxeter group $D_n$ is generated by $s,\sigma_1,\dots,\sigma_{n-1}$ with relations:
\begin{itemize}
	\item[$(1)$] $s^2=\sigma_i^2=e$;
	\item[$(2)$]$(\sigma_i \sigma_{i+1})^3=e$;
	\item[$(3)$]$(s \sigma_2)^3=e$;
	\item[$(4)$]$(s \sigma_i)^2=e \quad(i\neq 2)$;
	\item[$(5)$]$(\sigma_i \sigma_j)^2=e \quad(|i-j|>1)$.
\end{itemize}
and the corresponding generalized braid group $Br(D_n)$ is generated by $s,\sigma_1,\dots,\sigma_{n-1}$ with relations:
\begin{itemize}
	\item[$(1)$]$\sigma_i\sigma_{i+1}\sigma_i=\sigma_{i+1}\sigma_i\sigma_{i+1}$;
	\item[$(2)$]$s \sigma_2s=\sigma_2s\sigma_2$;
	\item[$(3)$]$s \sigma_i=\sigma_is \quad(i\neq 2)$;
	\item[$(4)$]$\sigma_i \sigma_j=\sigma_j\sigma_i \quad(|i-j|>1)$.
\end{itemize}
\begin{remark}\label{$D_n$ action}
We see from this point of view that the generators of $Br(D_n)$ correspond to these of $D_n$. Since $s$ and $\sigma_i$ are elements of $D_n$,
we can observe  how  $D_n$ acts freely on $F_{\mathbb{Z}_2}(\mathbb{C},n)$  in terms of $\sigma_i$ and $s$.  For each
$(z_1, ..., z_n)\in F_{\mathbb{Z}_2}(\mathbb{C},n)$, if
 $$\sigma_i(z_1, ..., z_{i-1}, z_i, z_{i+1}, z_{i+2}, ..., z_n)=(z_1, ..., z_{i-1}, z_{i+1}, z_{i}, z_{i+2}, ..., z_n)$$
  (i.e.,  $\sigma_i$ only permutes $i$-th and $(i+1)$-th coordinates of $z$) and
  $$s(z_1, z_2, z_3,  ..., z_n)=(-z_2, -z_1, z_3, ..., z_n)$$
 (i.e.,  $s$ just transfers $z_1$ to $-z_2$ and $z_2$ to $-z_1$), then we can verify easily that these transformations exactly satisfy the relations (1)--(5) in $D_n$.
\end{remark}

Now by Remark~\ref{$D_n$ action}, the generator $s$ in $Br(D_n)$ can be regarded as the class ${\bf b}{\bf b}_1{\bf b}$ in $\mathcal{B}_n^{\text{orb}}(\mathbb{C} ,\mathbb{Z}_2)$, and each $\sigma_i$ can be regarded as the class ${\bf b}_i$ in $\mathcal{B}_n^{\text{orb}}(\mathbb{C} ,\mathbb{Z}_2)$. Thus, we can define a map $$f: Br(D_n)\longrightarrow
\mathcal{B}_n^{\text{orb}}(\mathbb{C} ,\mathbb{Z}_2)$$ by $f(s)={\bf b}{\bf b}_1{\bf b}$ and $f(\sigma_i)={\bf b}_i$. A direct check shows that $f$ is a monomorphism.

\begin{proposition}
$Br(D_n)$ is isomorphic to a subgroup of $\mathcal{B}_n^{\text{orb}}(\mathbb{C} ,\mathbb{Z}_2)$.
\end{proposition}

\vskip .2cm
\noindent {\bf Comments.} The presentation of the generalized braid group $Br(D_n)$ is given easily by $D_n$ (see~\cite{B2, D}). At the same time, the presentation of the orbit braid group $B_n^{\text{orb}}(\mathbb{C},\mathbb{Z}_2)$ as described in Proposition~\ref{example-1} is much more geometric and intuitive, but not so easy to compute the group structure. Indeed, $B_n^{\text{orb}}(\mathbb{C}, \mathbb{Z}_2)$ is a bigger group than $Br(D_n)$.

\section{Extended fundamental groups of topological spaces}\label{extended fg}
Let $X$ be a path-connected topological space and let $\text{Homeo}(X)$ be the group given by all homeomorphisms from $X$ to itself. Here for convenience, $\text{Homeo}(X)$ is written as $\mathbb{G}(X)$ simply.  When $X$ is a connected smooth manifold, we may consider the group given by all diffeomorphisms from $X$ to itself.

\vskip .2cm
Fix a point $x_0$ in $X$ as a base point.
Choose a point $x$ in the orbit $\mathbb{G}(X)(x_0)$ at $x_0$, then we know from \cite{Br} that there uniquely exists a coset
$g\mathbb{G}(X)_{x_0}\in \mathbb{G}(X)/\mathbb{G}(X)_{x_0}$ such that for any $h\in g\mathbb{G}(X)_{x_0}$, $x=hx_0$, where $\mathbb{G}(X)_{x_0}$ is the isotropy subgroup at $x_0$. Throughout the following discussion, for each point $x\in \mathbb{G}(X)(x_0)$, we always fix a representative in its corresponding  coset, denoted by $g_x$, such that
$x$ has uniquely an expression $g_xx_0$. If
$x_0$ is of free orbit type, then $\mathbb{G}(X)_{x_0}=\{e\}$ so $g_x$ has only a unique choice.

\vskip .2cm

Now, for two paths $\alpha, \beta:I\longrightarrow X$  such that $\alpha(0)=\beta(0)=x_0$ and
$\alpha(1), \beta(1)\in\mathbb{G}(X)(x_0)$, we can produce a new path  $\alpha\circ (g_{\alpha(1)}\beta)$ in the usual  way:
 $$\alpha\circ (g_{\alpha(1)}\beta)(s)=\begin{cases}
 \alpha(2s) & \text{if} s\in [0, {1\over 2}]\\
 g_{\alpha(1)}\beta(2s-1) & \text{if} s\in [{1\over 2}, 1]
 \end{cases}
 $$
 with $\alpha\circ (g_{\alpha(1)}\beta)(0)=x_0$ and $\alpha\circ (g_{\alpha(1)}\beta)(1)=(g_{\alpha(1)} g_{\beta(1)})x_0$.
 Moreover, we define the operation $\bullet$ on the classes of paths up to homotopy relative to $\partial I$ as follows:
$$[\alpha]\bullet[\beta]=[\alpha\circ (g_{\alpha(1)}\beta)].$$
Note that $[\beta]\bullet[\alpha]=[\beta\circ(g_{\beta(1)}\alpha)]$.
Then we see that
$$\pi_1^E(X, x_0, \mathbb{G}(X)(x_0))=\big\{[\alpha]\big| \alpha:I\longrightarrow X \text{ with $\alpha(0)=x_0$ and $\alpha(1)\in\mathbb{G}(X)(x_0)$}\big\}$$
forms a group under the operation $\bullet$.

\begin{definition}
The group $\pi_1^E(X, x_0, \mathbb{G}(X)(x_0))$ is said to be the {\bf extended fundamental group} of $X$ at $\mathbb{G}(X)(x_0)$.
\end{definition}

Obviously, the extended fundamental group does depend upon the choices of representatives $g_x$ in the  expressions
$x=g_xx_0$,  $x\in \mathbb{G}(X)(x_0)$.

\vskip .2cm
The map $\Delta:\pi_1^E(X, x_0, \mathbb{G}(X)(x_0))\longrightarrow \mathbb{G}(X)/\mathbb{G}(X)_{x_0}$ defined by $[\alpha]\longmapsto g_{\alpha(1)}\mathbb{G}(X)_{x_0}$ is surjective since $X$ is path-connected, and the preimage of $\mathbb{G}(X)_{x_0}$ is exactly  the fundamental group $\pi_1(X, x_0)$.
Thus we have that

\begin{theorem} \label{efg-main}
There is the following short exact sequence
\begin{equation}\label{es}
1\longrightarrow \pi_1(X,x_0)\longrightarrow \pi_1^E(X, x_0, \mathbb{G}(X)(x_0))\longrightarrow \mathbb{G}(X)/\mathbb{G}(X)_{x_0}\longrightarrow 1.
\end{equation}
\end{theorem}
If $\mathbb{G}(X)_{x_0}$ is a normal subgroup of $\mathbb{G}(X)$, then (\ref{es}) is a short exact sequence in the sense of group homomorphisms;
otherwise, it is a short exact sequence in the sense of maps.
\begin{remark}\label{base point}
 We see from Theorem~\ref{efg-main} that up to isomorphism,  $\pi_1^E(X, x_0, \mathbb{G}(X)(x_0))$ does depend upon the choice of the base point $x_0$ since the isotropy subgroups at two different points may not be isomorphic. So the extended fundamental group  is not homotopy invariant.
  However, it is easy to see that $\pi_1^E$ is a functor from the category of topological spaces to the category of groups, so the extended fundamental group is homeomorphism invariant.
\end{remark}
 We can also use $\mathbb{G}(X)(x_0)$ to replace $\mathbb{G}(X)/\mathbb{G}(X)_{x_0}$ in the short exact sequence (\ref{es}) in Theorem~\ref{efg-main}.  Actually, consider the map $\Theta: \pi_1^E(X, x_0, \mathbb{G}(X)(x_0))\longrightarrow \mathbb{G}(X)(x_0)$ defined by $[\alpha]\longmapsto \alpha(1)=g_{\alpha(1)}x_0$. Clearly $\Theta$ is also surjective, and the preimage at $x_0$ is   the fundamental group $\pi_1(X, x_0)$. So we have the following short exact sequence
 $$1\longrightarrow \pi_1(X,x_0)\longrightarrow \pi_1^E(X, x_0, \mathbb{G}(X)(x_0))\longrightarrow \mathbb{G}(X)(x_0)\longrightarrow 1.$$

As  direct consequences of Theorem~\ref{efg-main}, we have that

\begin{corollary}\ \
\begin{itemize}
	\item[$(1)$] If $x_0$ is of free orbit type (i.e., $\mathbb{G}(X)_{x_0}=\{e\}$), then
$$\mathbb{G}(X)\cong \pi_1^E(X, x_0, \mathbb{G}(X)(x_0))/\pi_1(X,x_0),$$
giving a homotopy description of $\mathbb{G}(X)$.
	\item[$(2)$] If $x_0$ is a $\mathbb{G}(X)$-fixed point, then $\pi_1^E(X, x_0, \mathbb{G}(X)(x_0))$ degenerates into $\pi_1(X,x_0)$.
\end{itemize}
\end{corollary}

For a subgroup $H$ of $\mathbb{G}(X)$, $X$ is naturally regarded as the space with an effective action of $H$. Then the above procedure can still be carried out to define the extended fundamental group $\pi_1^E(X, x_0, H(x_0))$ of $X$ at $H(x_0)$
and to obtain the following short exact sequence
$$1\longrightarrow\pi_1(X,x_0)\longrightarrow \pi_1^E(X, x_0, H(x_0))\longrightarrow H/H_{x_0}\longrightarrow 1$$
or
$$1\longrightarrow\pi_1(X,x_0)\longrightarrow \pi_1^E(X, x_0, H(x_0))\longrightarrow H(x_0)\longrightarrow 1.$$
On the other hand, since $H(x_0)\subset \mathbb{G}(X)(x_0)$, without loss of generality we may assume that for each
$x\in H(x_0)\subset \mathbb{G}(X)(x_0)$, the choice of $h_x$ in the  expression $x=h_xx_0\in H(x_0)$ coincides with the choice of $g_x$ in the expression $x=g_xx_0\in \mathbb{G}(X)(x_0)$; in other words, $h_x=g_x$ for $x\in H(x_0)$.
 \vskip .2cm
 Thus,  $\pi_1^E(X, x_0, H(x_0))$ is a subgroup of $\pi_1^E(X, x_0, \mathbb{G}(X)(x_0))$.
Then it is easy to see the following result which  is  more general than Theorem~\ref{efg-main}.
\begin{theorem} \label{efg-main2}
The sequence
$$1\longrightarrow \pi_1^E(X,x_0, H(x_0))\longrightarrow \pi_1^E(X, x_0, \mathbb{G}(X)(x_0))\longrightarrow \mathbb{G}(X)(x_0)/H(x_0)\longrightarrow 1$$
is exact.
\end{theorem}

Finally we end this section with the following properties.
\begin{itemize}
	\item[${\bf (A)}$]
If  $H$  is finite, then  the projection
$p:X\longrightarrow X/H$ induces an epimorphism $$p_*: \pi_1^E(X, x_0, H(x_0))\longrightarrow \pi_1(X/H,  p(x_0)).$$
 \item[${\bf (B)}$] If  $H$  is finite and the action of $H$ on $X$ is free, then the projection
$p:X\longrightarrow X/H$ induces an isomorphism
$$\pi_1^E(X, x_0, H(x_0))\cong \pi_1(X/H, p(x_0)).$$
\item[${\bf (C)}$]
If $x_0$ is of free orbit type under the action of $\mathbb{G}(X)$, then  $X$ gives a direct system
$$\big\{ \pi_1^E(X, x_0, H(x_0)) \big| H\leq \mathbb{G}(X) \big\} $$
of the extend fundamental groups of $X$, such that the limit of this direct system is exactly $\pi_1^E(X, x_0, \mathbb{G}(X)(x_0))$.
\end{itemize}
\appendix
\section{Generalized braid group}\label{A}
Generalized braid groups, with respect to all finite Coxeter groups,  were introduced by Brieskorn \cite{B1} in the 1970's.  They are also called the Artin groups.
\vskip .2cm
 Following the terminology and notation of the paper by Vershinin \cite{VV}, let $V$ be an n-dimensional real vector space and let $W$ be a finite subgroup of $GL(V)$ generated by reflections. Let $\mathcal{M}$ be the set of hyperplanes such that $W$ is generated by the orthogonal reflections in the $M\in \mathcal{M}$. For any $\mathit{w}\in W$ and any $M \in \mathcal{M}$ we assume that $\mathit{w}(M)$ belongs to $\mathcal{M}$. Consider the complexification $V_\mathbb{C}$ of the space $V$ and the complexification $M_\mathbb{C}$ of $M \in \mathcal{M}$. Set $$Y_W=V_\mathbb{C}-\bigcup_{M\in \mathcal{M}}M_\mathbb{C}.$$
  Then $W$ acts freely on $Y_W$, and the orbit space of this action is denoted by $X_W=Y_W/W$. Then the fundamental group $\pi_1 (X_W)$ is called the braid group of action of $W$ on $V$,  denoted by $Br(V,W)$. The  fundamental group $\pi_1 (Y_W)$ is called the pure braid group of action of $W$ on $V$, denoted by $P(V,W)$.
\vskip .2cm
For a finite Coxeter group  $$ W=\langle\mathit{w}_1,\dots,\mathit{w}_k| (\mathit{w}_i \mathit{w}_j)^{\mathit{m}_{i,j}}=e, \mathit{m}_{i,j}=\mathit{m}_{j,i}\rangle,$$
the generalized braid group $Br(W)$ of $W$ is defined as the group with generators ${\mathit{w}_i}$ and relations
$$ \text{prod}(\mathit{m}_{i,j}; \mathit{w}_i, \mathit{w}_j)=\text{prod}(\mathit{m}_{j, i}; \mathit{w}_j, \mathit{w}_i)$$
where the symbol $\text{prod}(m;x,y)$ stands for the product $xyxy\cdots$ with m factors. By adding the relation $\mathit{w}_i^2=e$ to the above presentation we obtain a presentation of $W$. The following  theorem is due to Brieskorn~\cite{B2} and Deligne~\cite{D}.

\begin{theorem}[\cite{B2, D}]\ \
\begin{itemize}
	\item[$(1)$] The fundamental group $\pi_1(X_W)$ is isomorphic to the generalized braid group $Br(W)$.
	\item[$(2)$] The universal covering of $X_W$ is contractible, and hence $X_W$ is a space of $K(\pi; 1)$.
\end{itemize}
\end{theorem}
This theorem means that $X_W$ is the classifying space of  the generalized braid group $Br(W)$.
In addition, it is easy to see that $Y_W$ is also a space of $K(\pi; 1)$, so $Y_W$ is the classifying space of the generalized pure braid group $P(W)$ of $W$.


\end{document}